\documentclass[12pt]{article}

\usepackage{amssymb,amsthm}
\usepackage{amsmath}
\usepackage{latexsym}
\usepackage{amssymb}
\usepackage[dvips]{graphics}
\usepackage[dvips]{graphicx}
\newtheorem{theorem}{Theorem}[section]
\newtheorem{proposition}[theorem]{Proposition}
\newtheorem{lemma}[theorem]{Lemma}
\newtheorem{corollary}[theorem]{Corollary}
\newtheorem{definition}[theorem]{Definition}
\newtheorem{remark}[theorem]{Remark}
\newtheorem{example}[theorem]{Example}

\newcommand{\C}{{\mathbf C}}
\newcommand{\R}{{\mathbf R}}
\newcommand{\Z}{{\mathbf Z}}
\newcommand{\sing}{{\rm {Sing}}}
\renewcommand{\int}{\rm Int}
\newcommand{\E}{\mathbb E}

\newcommand{\sys}{{\rm {sys}}}
\newcommand{\PP}{{\mathbb P}}
\begin{document}

\title{Geometry and topology of random 2-complexes}         % Enter your title between curly braces
\author{A.E. Costa and M. Farber}        % Enter your name between curly braces
\date{June 23, 2014}          % Enter your date or \today between curly braces
\maketitle

%\abstract{We study random 2-dimensional complexes in the Linial - Meshulam model and prove that....}
{\small {\bf Abstract:} We study random 2-dimensional complexes in the Linial - Meshulam model and prove that 
the fundamental group of a random 2-complex $Y$ has cohomological dimension $\le 2$ if
the probability parameter satisfies $p\ll n^{-3/5}$. Besides, for $n^{-3/5} \ll p\ll n^{-1/2-\epsilon}$ the fundamental group 
$\pi_1(Y)$ has elements of order two and is of infinite cohomological dimension. 
We also prove that for $p\ll n^{-1/2-\epsilon}$ 
 the fundamental group of a random 2-complex has no $m$-torsion, for any given odd prime $m\ge 3$. 
We find a simple algorithmically testable criterion for a subcomplex of a random 2-complex to be aspherical; this implies that (for $p\ll n^{-1/2-\epsilon}$) any aspherical  subcomplex of a random 2-complex satisfies the Whitehead conjecture. 
We use inequalities for Cheeger constants and systoles of simplicial surfaces to analyse spheres and projective planes lying in random 2-complexes. 
Our proofs exploit the uniform hyperbolicity property of random 2-complexes (Theorem \ref{hyp}). 
%strengthens the results of E. Babson, C. Hoffman and M. Kahle.% \cite{BHK}. }

\section{Introduction and statements of the main results}

A model producing random simplicial complexes 
was recently suggested and studied by Linial and Meshulam ~\cite{LM} and further by 
Meshulam and Wallach ~\cite{MW}; it is a 
high-dimensional generalisation of one of the Erd\"os and R\'enyi models of random graphs  \cite{ER}.
In the Linial - Meshulam model one generates a random $d$-dimensional complex $Y$ by considering the full $d$-dimensional skeleton of the simplex 
$\Delta_n$ on vertices $\{1, \dots, n\}$ and retaining $d$-dimensional faces independently with probability $p$.  
The work of Linial--Meshulam and Meshulam--Wallach provides threshold functions for the vanishing of the $(d-1)$-st homology groups of random complexes with coefficients in a finite abelian group. Threshold functions for the vanishing of the $d$-th homology groups were subsequently studied by Kozlov \cite{Ko}. 

%An interesting class of closed smooth manifolds depending on a large number of random parameters arise as configuration spaces of mechanical linkages with bars of random  lengths,
%see  
%\cite{F1}, \cite{FK}. Although the number of homeomorphism type of these manifolds grows extremely fast, their topological characteristics can be predicted with high probability when the number of links tends to infinity. 

%
%Higher dimensional analogs of the aforementioned 
%Erd\H{o}s--R\'enyi model were recently suggested and studied by Linial--Meshulam in~\cite{LM}, and 
%Meshulam--Wallach in~\cite{MW}.
%In these models, one generates a random $d$-dimensional simplicial complex $Y$ by considering the full $d$-dimensional skeleton of the simplex 
%$\Delta_n$ on vertices $\{1, \dots, n\}$ and retaining $d$-dimen\-sional faces independently with probability $p$.  Note that in this construction $Y$ contains the $(d-1)$-dimensional 
%skeleton of $\Delta_n$. 
%The work of Linial--Meshulam and
%Meshulam--Wallach provides threshold functions for the vanishing of the
%$(d-1)$-st homology groups of random complexes with coefficients in a~finite
%abelian group. Threshold functions for the vanishing of the $d$-th homology groups were subsequently studied by Kozlov \cite{Ko}.

Significant progress in understanding the topology of random 2-complexes was made by Babson, Hoffman, and Kahle \cite{BHK} who 
investigated the fundamental groups of random 2-complexes.   
They showed that the fundamental group of a random 2-complex is either nontrivial and Gromov hyperbolic (for $p\ll n^{-1/2-\epsilon}$) or trivial (if $p^2 n -3\log n \to \infty$), a.a.s.\footnote{The symbol a.a.s. is an abbreviation of \lq\lq asymptotically almost surely\rq\rq, which means that the probability that the corresponding statement is valid tends to $1$ as $n\to \infty$.}  

In the paper \cite{CCFK} it was proven that a random 2-complex $Y$ collapses to a graph if $p\ll n^{-1}$ and therefore its fundamental group is free. 

The preprint \cite{ALLM} suggests an explicit constant $\gamma$ such that for any $c<\gamma$ a random $Y\in Y(n, c/n)$ is either collapsible to a graph or contains a  tetrahedron.

%In this paper we investigate asphericity properties of random 2-complexes. 
It is important to have a model producing random {\it aspherical} 2-dimensional complexes $Y$. 
A connected simplicial complex $Y$ is said to be aspherical if $\pi_i(Y)=0$ for all $i\ge 2$; this is equivalent to the requirement that the universal cover of $Y$ is contractible. 
For 2-dimensional complexes $Y$ the asphericity is equivalent to the vanishing of the second homotopy group $\pi_2(Y)=0$, or equivalently, that any continuous map 
$S^2\to Y$ is homotopic to a constant map. 
Random aspherical 2-complexes could be helpful for testing probabilistically the open problems of two-dimensional topology, such as the Whitehead conjecture. 
This conjecture stated by J.H.C. Whitehead in 1941 claims that a subcomplex of an aspherical 2-complex is also aspherical. 
Surveys of results related to the Whitehead conjecture can be found in \cite{B}, \cite{R}. 

The Linial-Meshulam model of random 2-complexes produces an aspherical complex $Y$ if $p\ll n^{-1}$, when $Y$ is homotopy equivalent to a graph, a.a.s.; 
however for 
$p\gg n^{-1}$ a random 2-complex is not aspherical since it contains a tetrahedron as a subcomplex, a.a.s. In \cite{CF} the authors 
studied {\it asphericable} 2-complexes, which can be made aspherical by deleting a few faces belonging to tetrahedra.  
The result of \cite{CF} states that random 2-complexes in the Linial - Meshulam model are asphericable in the range $p\ll n^{-46/47}$. 

In this paper we revisit the hyperbolicity theorem of Babson, Hoffman and Kahle \cite{BHK} having in mind two principal goals: firstly, we simplify and make more transparent the main ideas of the proof; secondly we observe that the arguments of \cite{BHK} give in fact a slightly stronger statement: a uniform lower bound on the isoperimetric constants of all subcomplexes of random complexes (see Theorem \ref{hyp} below and its full proof in \S \ref{app}).

%We employ a similar set of ideas centred around the Gromov's local-to-global principle. 
The key role in this paper plays Theorem \ref{lm4} which gives a topological classification of minimal cycles $Z$ with the property $\mu(Z)>1/2$. 
%The topological classification of minimal cycles implies as a corollary the homotopy classification of admissible 2-complexes (Lemma 4.1 from \cite{BHK}) which is a major step in the proof of hyperbolicity in \cite{BHK}. 
%Theorems \ref{lm4} and  \ref{hyp} are crucial for the proofs of the main results of this paper (in particular, for Theorem C). 

Now we state the major results obtained in this paper. 
\vskip 0.3cm

\noindent{\bf Theorem A.} [See Corollary \ref{cor35} and Proposition \ref{remcomb}] {\it If $p\ll n^{-3/5}$ then the fundamental group $\pi_1(Y)$ of a random 2-complex $Y\in Y(n,p)$ has cohomological dimension at most 2, a.a.s. 
In particular, $\pi_1(Y)$ is torsion free, a.a.s. Moreover, if 
$\frac{c}{n}< p\ll n^{-3/5}$,
where $c>3$, then ${\rm {cd}}(\pi_1(Y))=2$, a.a.s. }
\vskip 0.3 cm
The following theorem states that 2-torsion appears in the fundamental group $\pi_1(Y)$ once the probability parameter $p$ crosses the threshold $n^{-3/5}$. 
\vskip 0.3 cm
\noindent
{\bf Theorem B.} [See Theorem \ref{ordertwo}] {\it If for some $0<\epsilon <0.1$ the probability parameter $p$ satisfies 
$$ n^{-3/5}\ll p\ll n^{-1/2 -\epsilon}$$
then the fundamental group $\pi_1(Y)$ of a random 2-complex $Y\in Y(n,p)$ has nontrivial elements of order two and consequently ${\rm {cd}}(\pi_1(Y))=\infty$, a.a.s. }

\vskip 0.3cm
The next result complements Theorem B. \vskip 0.3cm

\noindent
{\bf Theorem C.} [See Theorem \ref{orderthree}] {\it Let $m\ge 3$ be an odd prime. If for some $\epsilon>0$ the probability parameter $p$ satisfies 
$p\ll n^{-1/2 -\epsilon}$ then, with probability tending to one as $n\to \infty$, a random 2-complex $Y\in Y(n,p)$ has the following property:
the fundamental group $\pi_1(Y')$ of any subcomplex $Y'\subset Y$ has no $m$-torsion. }

\vskip 0.3cm
The next statement describes aspherical subcomplexes of random 2-complexes. It is a strengthening of the main result of \cite{CF}. 
\vskip 0.3cm
\noindent{\bf Theorem D.} [See Theorem \ref{thmasph}] {\it Assume that 
$p\ll n^{-1/2-\epsilon}$
for a fixed $\epsilon >0$. Then a random 2-complex $Y\in Y(n,p)$ has the following property with probability tending to one as $n\to \infty$:
any subcomplex $Y'\subset Y$ is aspherical if and only if it contains no subcomplexes $S\subset Y'$ with at most $2\epsilon^{-1}$ faces which are homeomorphic either to 
the sphere $S^2$, the projective plane ${\mathbf {RP}}^2$ or the complexes $Z_2, Z_3$ shown on Figure \ref{fig6}. 
}
\vskip 0.3cm

Theorem D implies the following result which can be viewed as probabilistic confirmation of the Whitehead Conjecture. 
\vskip 0.3 cm

\noindent{\bf Corollary E.} [The Whitehead Conjecture] {\it Assume that 
$p\ll n^{-1/2-\epsilon}$ 
for a fixed $\epsilon >0$. Then a random 2-complex $Y\in Y(n,p)$ has the following property with probability tending to one as $n\to \infty$:
any aspherical subcomplex $Y'\subset Y$ satisfies the Whitehead Conjecture, i.e. any subcomplex $Y''\subset Y'$ is also aspherical. }

\vskip 0.3 cm
Proposition \ref{remcomb} describes the second Betti number of random aspherical 2-complexes obtained from a random 2-complex $Y\in Y(n,p)$ by deleting a sequence of faces. It follows that these complexes have non-free fundamental groups. 

\vskip 0.3 cm

The following corollary states that a random 2-complex in the range $p\ll n^{-1/2-\epsilon}$ contains no tori and no large spheres and projective planes.
\vskip 0.3 cm
\noindent{\bf Corollary F.} [see Corollaries \ref{cortorus}, \ref{corsphere}, \ref{noprojplanes}] {\it Suppose that $p\ll n^{-1/2-\epsilon}$, where $\epsilon>0$ is fixed. Then: 
\begin{enumerate}
  \item[(a)] A random 2-complex $Y\in Y(n,p)$ contains no subcomplex 
homeomorphic to the torus $T^2$, a.a.s.
  \item[(b)] A random 2-complex $Y\in Y(n,p)$ contains no subcomplex with more than $2\epsilon^{-1}$ faces which is homeomorphic to the sphere $S^2$, a.a.s. 
  \item[(c)] A random 2-complex $Y\in Y(n,p)$ contains no subcomplex with more than $\epsilon^{-1}$ faces which is homeomorphic to the real projective plane ${\mathbf {RP}}^2$, a.a.s. 
\end{enumerate}  
}

Note that under the assumptions of Corollary F  (as follows from the results of \cite{CCFK})
any triangulated sphere with $\le 2\epsilon^{-1}$ faces is contained in a random 2-complex as a subcomplex and, besides, any triangulated real projective plane with $\le \epsilon^{-1}$ faces is contained in a random 2-complex, a.a.s. Thus Corollary F lists all spheres and projective planes which cannot be found as subcomplexes of random complexes with probability tending to one. 

We know that for $p=n^\alpha$ with $\alpha > -1/2$ a random 2-complex $Y\in Y(n,p)$ is simply connected and hence it has homotopy type of a wedge of 2-spheres, see \cite{BHK}. Thus, a significant change in geometry and topology of random 2-complexes happens around $p=n^{-1/2}$, and an outstanding challenge for the research field is to analyse this event in more detail. 

Figure \ref{figcd} illustrates the behaviour of the cohomological dimension of the fundamental group $\pi_1(Y)$ of a random 2-complex $Y\in Y(n,p)$. 
For simplicity we assume here that $p=n^\alpha$ where $\alpha$ is a real parameter, $\alpha\in (-\infty, 0)$. 
\begin{figure}[h]
\centering
\includegraphics[width=0.8\textwidth]{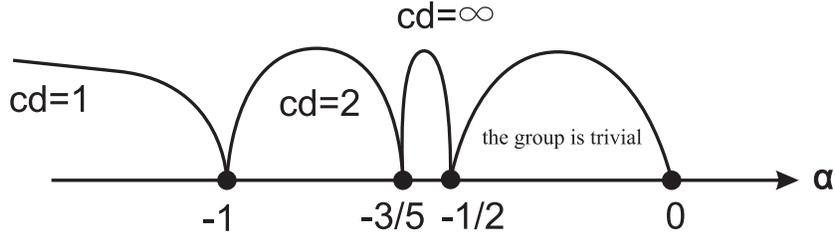}
\caption{The cohomological dimension ${\rm cd} ={\rm cd}(\pi_1(Y))$ of the fundamental group of a random 2-complex $Y\in Y(n,p)$ where $p=n^\alpha$.}\label{figcd}
\end{figure}
The behaviour of the torsion in the fundamental group $\pi_1(Y)$ of a random 2-complex $Y\in Y(n,p)$ 
is illustrated by Figure \ref{2torsion-diagram}. Again, we assume here that $p=n^\alpha$, where $\alpha\in (-\infty, 0)$. We also know that $\pi_1(Y)$ has no $m$-torsion, 
where $m$ is an odd prime, for $p=n^\alpha$ with $\alpha\not= -1/2$. 
\begin{figure}[h]
\centering
\includegraphics[width=0.8\textwidth]{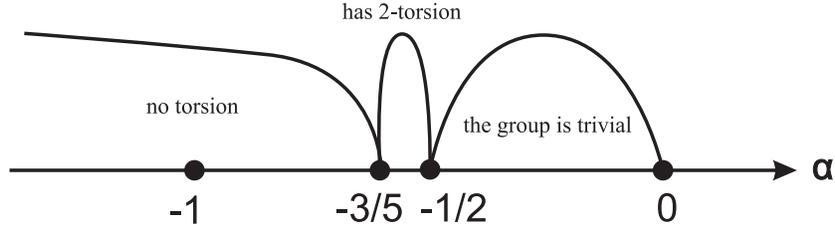}
\caption{Torsion in the fundamental group of a random 2-complex $Y\in Y(n,p)$ where $p=n^\alpha$.}\label{2torsion-diagram}
\end{figure}

A few words about terminology we use in this paper.  
By a {\it 2-complex} we understand a finite simplicial complex $Z$ of dimension $\le 2$. 
The $i$-dimensional simplexes of a 2-complex are called {\it vertices} (for $i=0$), {\it edges} (for $i=1$) and {\it faces} (for $i=2$). 
A 2-complex is said to be {\it pure} if any vertex and any edge are incident to a face. For an edge $e\subset Z$ we denote by $\deg e=\deg_Z(e)$ the degree of $e$, i.e. the number of faces of $Z$ containing $e$. The union of edges of degree one is denoted $\partial Z$ and is called the {\it boundary} of $Z$. 
We say that a 2-complex $Z$ is {\it closed} if $\partial Z=\emptyset$. 

We use the notations $P^2$ and ${\mathbf {RP}}^2$ for the real projective plane. 

For a 2-complex $X$ we denote by $\mu(X)$ the ratio $v/f$ where $v=v(X)$ is the number of vertices  and $f = f(X)$ is the number of faces in $X$. 
The symbol $\tilde \mu(X)$ denotes 
$$\tilde \mu(X) = \min\mu(S)$$ where $S\subset X$ runs over all subcomplexes. More information about the properties of the invariants $\mu(X)$ and $\tilde \mu(X)$ and their relevance to the containment problem for random complexes can be found in 
\cite{BHK}, \cite{CCFK}. For convenience of the reader we shall state here one of the results which will be used frequently in the paper, compare 
Theorem 15 in \cite{CCFK}. 

\begin{lemma} \label{containment} Let $\mathcal F$ be a finite collection of finite simplicial complexes. 
Denote $\alpha= \min\{ \tilde \mu(X), X\in \mathcal F\}$, and 
$\beta= \max\{ \tilde \mu(X), X\in \mathcal F\}$. %Consider a random 2-complex $Y\in Y(n,p)$. 
Let $A_n^{\mathcal F}\subset Y(n,p)$ denote the set of complexes 
$Y\in Y(n,p)$ which contain each of the complexes $X\in \mathcal F$ as a simplicial subcomplex. 
Let $B_n^{\mathcal F}\subset Y(n,p)$ denote the set of complexes 
$Y\in Y(n,p)$ which contain at least one of the complexes $X\in \mathcal F$ as a simplicial subcomplex. Clearly $A_n^{\mathcal F} 
\subset B_n^{\mathcal F}$. Then one has:

(1) If $p\ll n^{-\beta}$ then $\mathbb P(B_n^{\mathcal F}) \to 0$ as $n\to \infty$;

(2) If $p\gg n^{-\alpha}$ then $\mathbb P(A_n^{\mathcal F}) \to 1$ as $n\to \infty$.
\end{lemma}

In other words, (1) states that for small $p$ none of the complexes $X\in \mathcal F$ is embeddable into a random 2-complex, a.a.s.; statement (2) states that for large $p$ 
any complex $X\in \mathcal F$ is embeddable into a random 2-complex, a.a.s.

%\section{Hyperbolicity and isoperimetric constants of random simplicial complexes}

%The main result of this section is Theorem \ref{hyp} which strengthens the main theorem of \cite{BHK}. 
%We not only show that the fundamental group of a random 2-complex $Y\in Y(n,p)$ is hyperbolic (assuming that $p\ll n^{-1/2-\epsilon}$) but also prove that (with probability tending to one)
%there is a uniform lower bound on the isoperimetric constants of all subcomplexes $Y'\subset Y$. 
%
%The main ingredients of the proof of Theorem \ref{hyp} are (a) the local-to-global principle of Gromov (Theorem \ref{localtoglobal}), (b) topological classification of minimal cycles (Theorem \ref{lm4}) and (c) the uniform lower bound on the isoperimetric constants of all complexes with $\tilde\mu(X)>1/2+\epsilon$, see Theorem \ref{uniform}. Theorem \ref{uniform} appears in \cite{BHK} as Lemma 3.5; we presented here a modified proof of this statement for completeness. 

\section{Topology of minimal cycles} 

In this section we classify topologically the minimal cycles satisfying the condition 
$\mu(Z)>1/2$; this result will be important for our study of aspherical subcomplexes and the Whitehead conjecture in \S \ref{secwh}. 
 
     % Enter section title between curly braces
If $v\in Z$ is a vertex of a 2-complex $Z$ then the link of $v$ is a graph; we denote by $\sing(v)=\sing_Z(v)$ the number of connected components of the link of $v$ minus one.

Let $Z$ be a finite pure 2-complex with the set of vertices $V(Z)$. The closure of a connected component of the complement $Z-V(Z)$ %(where $V(Z)$ denotes the set of vertices of $Z$) 
is called {\it a strongly connected component} of $Z$. A finite pure complex $Z$ is called {\it strongly connected} if there is a unique strongly connected component. 
Note that if $Z$ is path-connected and $\sing(v)=0$ for any vertex $v\in V(Z)$ then $Z$ is strongly connected.

Given a pure 2-complex $Z$, there is a canonically defined 2-complex $\tilde Z$ 
with a natural surjective map $\tilde Z \to Z$, which is 1-1 on $\tilde Z - V(\tilde Z)$, such that $\sing_{\tilde Z}(w)=0$ for any vertex $w$ of $\tilde Z$. 
The complex $\tilde Z$ is obtained by multiplying every vertex of $Z$ as many times as there are connected components in the link of this vertex 
(in other words, we cut open any vertex $v$ with $\sing(v)>0$ along the cut point). 
Clearly, $\tilde Z$ is strongly connected if $Z$ is strongly connected. 
In this case one observes that $Z$ is homotopy equivalent to the wedge sum 
$\tilde Z \vee S^1 \vee S^1 \vee \dots \vee S^1,$
where the number of $S^1$ summands equals  $\sum_{v\in V(Z)}\sing(v).$
This gives the following corollary:

\begin{corollary}\label{cor1} For any pure strongly connected 2-complex 
$Z$ one has $$b_1(Z) = b_1(\tilde Z) + \sum_{v\in V(Z)} \sing(v)\ge \sum_{v\in V(Z)} \sing(v).$$ 
In particular, if $Z$ has at least one vertex with $\sing(v)\ge 1$ then $b_1(Z) \ge 1$. 
\end{corollary}

\begin{definition}\label{def1} A finite pure 2-complex $Z$ is said to be a minimal cycle if $b_2(Z)=1$ and for any proper subcomplex $Z'\subset Z$ one has $b_2(Z')=0$. 
\end{definition}

Clearly a minimal cycle must be closed (i.e. $\partial Z=\emptyset$) and strongly connected. Indeed, if $Z_1, \dots, Z_k$ are strongly connected components of $Z$, $k\ge 2$, then 
$b_2(Z) = \sum b_2(Z_i)$ and hence exactly one component $Z_i$ has the second Betti number 1. Removing a face from another component $Z_j$, $j\not=i$, will not affect the second Betti number, which gives a contradiction.  

The following easy statement will be useful later.

\begin{lemma}\label{lemmanew} Let $\phi: Z_1 \to Z_2$ be a simplicial map between finite 2-complexes such that $b_2(Z_1)\ge 1$ and $Z_2$ is a minimal cycle. Suppose that $\phi$ maps bijectively the set of faces of $Z_1$ onto the set of faces of $Z_2$ (and in particular, faces of $Z_1$ do not degenerate under $\phi$). Then $Z_1$ is also a minimal cycle; in particular $Z_1$ is strongly connected. 
\end{lemma}
\begin{proof} Assuming that $Z_1$ is not a minimal cycle we may find a proper subcomplex $Z_1'\subset Z_1$ with $b_2(Z'_1)=1$. Then our assumptions imply that the subcomplex 
$\phi(Z_1')\subset Z_2$ is proper and carries a nontrivial 2-cycle, i.e. $b_2(\phi(Z_1'))\ge 1$ contradicting the minimality of $Z_2$. 
\end{proof}

\begin{lemma}\label{lm2}
Let $Z$ be a minimal cycle and let $D\subset Z$ be a subcomplex homeomorphic to a 2-disc such that the interior $\int(D)$ is open in $Z$. Denote by $Z'=Z- \int(D)$, the result of deleting the interior of $D$. Then the 2-complex $Z'$ is strongly connected. 
\end{lemma}
\begin{proof}
Assume the contrary, i.e. let $Z'$ have at least two strongly connected components $Z'_1, Z'_2, \dots$. Then the quotient $Z/D$ is homotopy equivalent to 
the wedge sum of the spaces 
$Z'_i/(Z'_i\cap D)$ (where $i=1, 2, \dots$) and some number of circles. 
It is obvious that in the case $Z_i'\cap D=\partial D$ for some $i$ the complex $Z'$ is strongly connected. Thus we may assume that each intersection $Z_i'\cap D$ is a proper subcomplex of the circle $\partial D$, i.e.  $Z_i'\cap D$ is homeomorphic to a disjoint union of intervals. 
Hence the quotient $Z_i'/(Z_i'\cap D)$ is homotopy equivalent to the wedge sum of $Z_i'$ with several copies of $S^1$ which gives 
$b_2(Z_i'/(Z_i'\cap D))= b_2(Z_i')$. 
Thus we have 
\begin{eqnarray}\label{colap1}b_2(Z) = b_2(Z/D) = \sum_i b_2(Z'_i/(Z'_i\cap D))= \sum_i b_2(Z'_i).\end{eqnarray}
From (\ref{colap1}) it follows that among the numbers $b_2(Z'_i)$ exactly one equals 1 and the others vanish. This contradicts the assumption that $Z$ is a minimal cycle since it contains a proper subcomplex $Z'_i\subset Z$ with $b_2(Z'_i)=1$.  
\end{proof}

\begin{theorem}\label{lm4} {\rm [Classification of minimal cycles]} Let $Z$ be a minimal cycle satisfying $$\mu(Z) > 1/2.$$ Then $Z$ is homeomorphic to one of the following 2-complexes
$Z_1, Z_2, Z_3, Z_4,$
where $Z_1=S^2$; $Z_2$ is the quotient of $S^2$ with two distinct points identified; $Z_3$ is the quotient of $S^2$ with two adjacent arcs identified; and $Z_4$ is defined as 
$P^2\cup \Delta^2$, where $\partial \Delta^2= P^2\cap \Delta^2$ equals the curve $P^1\subset P^2$. 
\end{theorem}
\begin{figure}[h]
\centering
\includegraphics[width=0.3\textwidth]{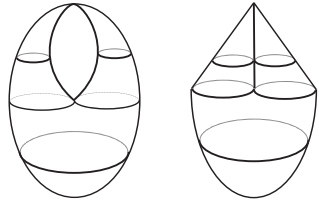}
\caption{Minimal cycles $Z_2$ (left) and $Z_3$ (right).}\label{fig6}
\end{figure}

\begin{proof} Recall that for an edge $e\subset Z$ one denotes by $\deg e$ the degree of $e$, i.e. the number of faces containing $e$. Using the following formula
\begin{eqnarray}\label{mu}\mu(Z) = \frac{1}{2} + \frac{2\chi(Z) +L(Z)}{2f(Z)},\end{eqnarray}
where $$L(X) = \sum_e(2-\deg(e)),$$ the sum is taken over all edges $e$ of $X$,  (see (7), (8) in \cite{CF}),
we see that the condition $\mu(Z) >1/2$ translates into 
\begin{eqnarray}\label{mu12}2\chi(Z) + L(Z)>0.\end{eqnarray}

Note that $L(Z) \le 0$ since $Z$ is pure and closed. We consider now a few special cases:

Case A: $L(Z) =0$. Then $Z$ is a pseudo-surface\footnote{{\it A pseudo-surface} is a 2-complex which is strongly connected and has the property that each edge is incident to exactly 2 faces. 
}, i.e. each edge has degree 2. 
The link of every vertex is a disjoint union of several circles. In other words $Z$ can be obtained from a connected surface $\tilde Z$ by identifying several vertices. By (\ref{mu12}) one has $\chi(Z)>0$ and hence $Z$ is either the sphere $Z_1$ or the sphere with two points identified $Z_2$. 
Hence we see that in the case A the complex $Z$ is homeomorphic either to $Z_1$ or $Z_2$. 

Case B: $L(Z)=-1$. Let us show that this is impossible. Indeed, in this case all edges have degree 2 except one edge which has degree 3. If $e$ is this edge of degree 3 then 
for one of the incident vertices $v$ one considers the link of $v$; it is a graph in which all vertices except one have degree 2 and one vertex has degree 3. This is impossible since 
the sum of degrees of all vertices in a graph is always even (as it equals twice the number of edges of the graph). 

If $b_1(Z) \ge 1$ then  $\chi(Z) \le 1$ and (\ref{mu12}) gives $L(Z) \ge -1$, i.e. $L(Z)=0$ or $L(Z)=-1$. Thus minimal cycles satisfying $b_1(Z)\ge 1$ we are in one of the cases A or B considered above. We shall assume below that $b_1(Z)=0$. 

By Corollary \ref{cor1} the condition $b_1(Z)=0$ implies that the link of every vertex of $Z$ is connected.

The assumptions $b_2(Z)=1$ and $b_0(Z)=1$ imply $\chi(Z) = 2$ and hence 
$\mu(Z) > 1/2$ is equivalent to $L(Z) \ge -3$. 
Thus, we have to consider the following four special cases $L(Z)=0,-1, -2, -3$ with the two first analysed above. 

Case C: $L(Z)=-2$. There are two possibilities: either $Z$ contains two edges of degree 3 or one edge of degree 4. Consider the first possibility: let $e, e'$ be two edges of degree 3 and let $v$ be a vertex incident to $e$ but not to $e'$. Then the link of $v$ is a graph with all  vertices of degree 2 except one which has degree 3, which is impossible. 

\begin{figure}[h]
\centering
\includegraphics[width=0.7\textwidth]{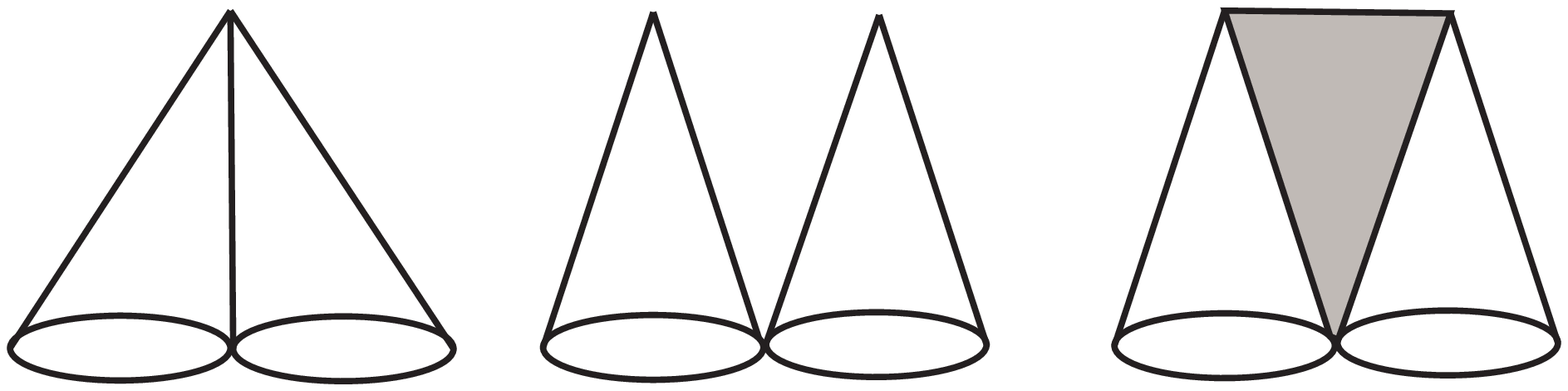}
\caption{Cones.}\label{fig3}
\end{figure}
Now suppose that $Z$ contains an edge $e$ of degree 4 and all other edges of $Z$ have degree 2. Denote by $v, w$ the endpoints of $e$. The links of $v$ and $w$ are connected graphs with one vertex of degree 4 and all other vertices of degree 2.  Hence these links are homeomorphic to the figure eight $\Gamma_8$ (the wedge sum of two circles). The regular neighbourhood of $v$ is the star of $v$ which is homeomorphic to the cone over $\Gamma_8$. 
Replacing the cone $C(\Gamma_8)$ by the wedge of two cones over the base circles (see Figure \ref{fig3}, middle) gives a pseudo surface $Z'$ on which the original edge $e$ of degree 4 is represented by two adjacent edges $e_1, e_2$ of degree 2. Clearly $Z'$ is homotopy equivalent to $Z$: indeed, attaching a triangle to $Z'$ as shown on 
Figure \ref{fig3} (right) gives a homotopy equivalent space as the triangle can be collapsed along the free edge; on the other hand collapsing the whole triangle gives a space homotopy equivalent to $Z$. 
By Lemma \ref{lemmanew} the complex $Z'$ is a minimal cycle and therefore it is strongly connected. 
Thus, $Z'$ is a pseudo-surface with $b_1(Z')=0$ and $b_2(Z')=1$. Using Corollary \ref{cor1} we see that $Z'$ is nonsingular and hence is homeomorphic to $S^2$. 
The complex $Z$ is obtained from this sphere by folding two adjacent edges $e_1, e_2$. Thus, $Z$ is homeomorphic to $Z_3$. 

Case D: $L(Z)=-3$. There are three possibilities: either (1) $Z$ contains one edge $e$ of degree 5 and all other edges have degree 2, or (2) there are two edges $e_1, e_2$ with 
$\deg(e_1)= 3$, and $\deg(e_2) = 4$ and all other edges have degree $2$, and (3) there are 3 edges $e_i$ with $\deg(e_i)=3$ for $i=1, 2, 3$ and all other edges of $Z$ have degree 2. In the case (1), if $v$ is a vertex incident to $e$, then the link of $v$ is a graph with one vertex of degree $5$ and all other vertices of degree $2$ which is impossible. Similarly, the case (2) is impossible; one may apply these arguments to the vertex of $e_1$ which is not incident to $e_2$. The same reasoning shows that in the case (3) the only possibility is that the edges $e_1, e_2, e_3$ form a triangle. 

Below we assume that the edges of degree three $e_1, e_2, e_3$ form a triangle with vertices $v_1, v_2, v_3$, i.e. the edge $e_1$ has vertices $v_1, v_2$, the edge $e_2$ has vertices $v_2, v_3$ and $e_3$ has vertices $v_3, v_1$.

Consider the link of $v_i$ in complex $Z$. It is a graph with two vertices of degree three and with all other vertices of degree 2. Hence the link is homeomorphic to 
one of the graphs $\Gamma_1, \Gamma_2$ shown on Figure \ref{fig1}. 
\begin{figure}[h]
\centering
\includegraphics[width=0.4\textwidth]{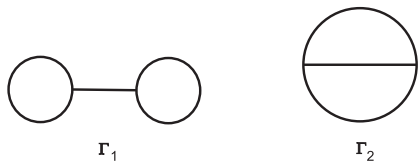}
\caption{Possible links of $v_i$.}\label{fig1}
\end{figure}
Let us show that the graph $\Gamma_1$ is in fact impossible. Indeed, suppose that the link of $v_i$ is homeomorphic to $\Gamma_1$. The arc connecting two triple points of $\Gamma_1$ corresponds to a sequence of 2-faces $\sigma_1, \dots, \sigma_r$ such that $\sigma_j\cap \sigma_{j+1}$ is an edge containing $v_i$. The union 
$\sigma=\sigma_1\cup \dots \cup \sigma_r$ is a subcomplex homeomorphic to a disc. Consider $Z'=Z-\int(\sigma)$. By Lemma \ref{lm2}, the complex $Z'$ is strongly connected. We have $\chi(Z')=1$ and $b_2(Z')=0$ and hence $b_1(Z')=0$. This contradicts Corollary \ref{cor1} since the link of $v_i$ in $Z'$ has two connected components.

Suppose that there is a 2-simplex $\sigma$ in $Z$ having the vertices $v_1, v_2, v_3$. Then removing the interior of this simplex we obtain a 
pseudo surface $Z'$ (i.e. $L(Z')=0$) with $b_1(Z')=b_2(Z')=0$. 
By Lemma \ref{lm2}, the complex $Z'$ must be strongly connected. 
Then $\sing(v)=0$ for every vertex $v\in Z'$, since otherwise $b_1(Z')>0$, see Corollary \ref{cor1}. Hence $Z'$ is a genuine closed surface and since we know the Betti numbers of $Z'$, clearly $Z'$ must be homeomorphic to the real projective plane. Thus $Z$ is homeomorphic to the union of $P^2$ and a triangle $\Delta^2$. The intersection $P^2\cap \Delta^2$ is a simple closed curve. 
If this curve is null-homologous on $P^2$ then 
this curve must separate $P^2$ into two connected components one of which must be homeomorphic to a 2-disc and together with $\Delta^2$ this disc would form proper a 2-cycle in $Z$ contradicting the minimality of $Z$ (see Definition \ref{def1}). Hence, the curve $P^2\cap \Delta^2$ is homologous to $P^1\subset P^2$. Therefore $Z$ is homeomorphic to $Z_4$.

Next we consider the case that there is no 2-simplex with sides $e_1, e_2, e_3$. 
Let us label the three 2-faces incident to $e_1$ by the symbols $1,2,3$. 
In the link of $v_2$ the edge $e_1$ corresponds to one of the vertices of degree three and the symbols $1, 2, 3$ are now associated with the three arcs (see the graph $\Gamma_2$ on Figure \ref{fig1}) 
incident to this vertex. Viewing the other triple point of the link of $v_2$ we obtain a labelling by the symbols $1, 2, 3$ of the three triangles incident to the edge $e_1$. Thus we see that a labelling by the symbols $1, 2, 3$ of the three triangles containing an edge $e_i$ determines a labelling of the following edge $e_{i+1}$. 
Performing this process in the following sequence $e_1\to e_2\to e_3\to e_1$ we obtain a permutation $\tau$ of the symbols $1, 2, 3$. 
We denote by $k \, (=1, 2, 3)$ the number of orbits of the action of $\tau$ on $\{1, 2, 3\}$.

Case D1: Suppose that $k=1$. Let $c=\sum_j n_j \sigma_j$ be a nontrivial 2-cycle of $Z$ with integral coefficients $n_j$; here the symbols 
$\sigma_i$ denote distinct oriented simplexes of $Z$. Due to our minimality assumption, $n_i\not=0$ for any $i$.  If two simplexes $\sigma_i$ and $\sigma_{i'}$ have a common edge  of degree two then $n_i = \pm n_{i'}$.  The complement 
$Z-E$ (where $E$ is the union of the triple edges $e_1, e_2, e_3$) is strongly-connected since $Z$ is strongly connected and $k=1$, i.e. one may pass from any face incident to $E$ to any other face incident to $E$ by jumping across edges not in $E$. Thus, any two 2-simplexes $\sigma, \sigma'$ of $Z$ can be connected by a chain $\sigma=\sigma_1, \sigma_2, \dots, \sigma_r=\sigma'$ such that each $\sigma_i$ and $\sigma_{i+1}$ have a common edge of degree two. 
Hence we obtain that $n_i=\pm n_j$ for any $i, j$. This leads to a contradiction 
with $\partial c=0$ since 
near an edge  $e_r$ of degree three we will have three 2-simplexes, each contribution the same amount, up to a sign. Thus, Case D1 is impossible. 

Case D2: Suppose that $k=2$ or $k=3$. Then near the union of the triple edges $E$ the complex $Z$ is homeomorphic either to $Y\times S^1$ (where $Y$ is the graph representing the letter Y) or to the union of the M\"obius band $M$ and the cylinder $N=[0,1]\times S^1$ such that $M\cap N=0\times S^1$ is the central circle of the 
M\"obius band. In either case we may cut $Z$ along the triple circle $E$, separating a cylinder from the rest, such that the result is a surface $\hat Z$ with boundary. 
Points of $E$ will be represented on $\hat Z$ by two circles, one being the base of the cylindrical part and another the middle circle of a cylinder or a M\"obius band. 
If $\hat Z$ is connected then it is homotopy equivalent to a graph and hence $\chi(\hat Z)=\chi(Z) \le 1$. Since $L(Z)=-3$, we see that (\ref{mu12}) is violated. 
Therefore 
the surface $\hat Z$ must have two connected components $\hat Z_1\sqcup \hat Z_2$, one of which (say, $\hat Z_1$) is a closed surface, while the other $\hat Z_2$ has one boundary circle. The minimality condition implies that the closed surface $\hat Z_1$ must be non-orientable.
Thus we have $2= \chi(Z)=\chi(\hat Z)=\chi(\hat Z_1) + \chi(\hat Z_2)$ and $\chi(\hat Z_1)\le 1$, $\chi(\hat Z_2) \le 1$ which implies that $\chi(\hat Z_1)=1=\chi(\hat Z_2)$. Therefore, $\hat Z_1$ is the real projective plane and $\hat Z_2$ is a disk. 
We obtain that $Z$ is the union of the projective plane and a disk intersecting along the triangle $E$. 
Clearly $E$ is not null-homologous in $\hat Z_1$ (since otherwise it would cut the projective plane into two connected components one of which would be orientable and together with $\hat Z_2$ it would make a cycle, contradicting minimality of $Z$). %\marginpar{Improve}

Clearly the pair $(\hat Z_1, E)$ is homeomorphic to $(P^2, P^1)$, and $Z$ is homeomorphic to $Z_4$. 

This completes the proof. 
\end{proof}

%Consider an internal point $v$ on the edge $e_1$; the three edges incident to $e_1$ 
%
%
%In this case show that $Z$ is the union of a disc with boundary $e_1, e_2, e_3$ and  a subset $Z'\subset Z$ which is a pseudo-surface which has $\chi(Z')=1$. We also have that $b_1(Z')=0$. This implies that $Z'$ is the projective plane. Therefore $Z$ is homeomorphic to the union of $P^2$ and the disc. If the boundary of the disc is null-homotopic then the complex $Z$ is not minimal. 
%

\begin{corollary}\label{cor2} In any minimal cycle $Z$ with $\mu(Z) >1/2$ there exists a 2-face $\sigma$ such that the boundary curve $\partial \sigma$ is null-homotopic in the complement $Z - \int(\sigma)$. 
\end{corollary}

\begin{lemma}\label{btwo0}
Let $X$ be a finite closed strongly connected pure 2-complex with $b_2(X)=0$. 
Then either $\mu(X)\le 1/2$ or $X$ is homeomorphic to the projective plane ${\bf RP}^2$. 
\end{lemma}
\begin{proof} Assume that $\mu(X)>1/2$; this is equivalent (using formula (\ref{mu})) to 
\begin{eqnarray}\label{ineql}
{\mathcal L}(X)\equiv 2\chi(X)+L(X) =2-2b_1(X)+L(X) >0.
\end{eqnarray}
Besides, we have $L(X)\le 0$, since $X$ is closed. The inequality (\ref{ineql}) implies that $L(X)$ can be either $-1$ or $0$. In the proof of Theorem \ref{lm4} 
(see Case B) we showed
that $L(X)=-1$ is impossible. Hence, $L(X)=0$, i.e. $X$ is a pseudo-surface. 
Inequality (\ref{ineql}) now implies that $b_1(X)=0$. 
Using Corollary \ref{cor1} we obtain that $X$ is a genuine surface without singularities. 
Thus $X$ is a compact nonorientable surface %and using formula (9) on p. 132 of \cite{CCFK}, 
with $b_1(X)=0$, i.e. 
%we obtain that the genus of $X$ equals $1$ (since otherwise $\mu(X)\le 1/2$), i.e. 
$X$ is homeomorphic to the projective plane. 
\end{proof}

We shall say that a 2-complex $X$ is {\it  an iterated wedge of projective planes} if it can be represented as a union of subcomplexes 
$X=\cup_{i=1}^n A_i$ where each $A_j$ is homeomorphic to a projective plane and the intersection 
$$A_j\cap (\cup_{i=1}^{j-1}A_i)$$ is a single point set for $j=1, \dots, n$. 

\begin{lemma}\label{lmstrong} Let $X$ be a finite pure connected 2-complex with $b_1(X)=0$. Then $X$ is iterated wedge of its strongly connected components. 
\end{lemma}
\begin{proof} We shall use the following remark. 
Let $A\cup B=X$ be two connected subcomplexes which are unions of strongly connected components of $X$ and such that every strongly connected component of $X$ is contained either in $A$ or in $B$ but not in both 
$A$ and $B$. Then the intersection $A\cap B$ is a finite set of vertices and using $b_1(X)=0$ in the Mayer - Vietoris sequence implies 
that $\tilde H_0(A\cap B)=0$, i.e. $A\cap B$ is a single point. 

Let $A_1$ be a strongly connected component of $X$. Using the connectedness of $X$ and the previous remark, we may find another strongly connected component 
$A_2$ of $X$ such that $A_1\cap A_2$ is a single point. Inductively, we may find a sequence $A_1, \dots, A_r$ of strongly connected components of $X$ such that 
$X=A_1\cup \dots\cup A_r$ and for any $j$ the intersection 
$$A_j\cap (\cup_{i=1}^{j-1} A_i)$$
is a single point. Hence $X$ is an iterated wedge of its strongly connected components. 
\end{proof}

A version of Lemma \ref{btwo0} with the words \lq\lq strongly connected\rq\rq\, replaced by \lq\lq connected\rq\rq\, reads as follows.

\begin{lemma} \label{btwo01}
Let $X$ be a finite closed connected pure 2-complex with $b_2(X)=0$ and $\mu(X)>1/2$. 
Then $X$ is homeomorphic to an iterated wedge of projective planes. 
\end{lemma}
\begin{proof} We will use induction on the number of strongly connected components $k$ of $X$. Lemma \ref{btwo0} covers the case $k=1$. 
Assume now that Lemma \ref{btwo01} has been proven for all complexes $X$ with less than $k$ strongly connected components. 

If $X$ has $k$ strongly connected components and $b_2(X)=0$, $\mu(X) >1/2$ then $b_1(X) =0$ (as follows from the inequality (\ref{mu12}) combined with 
$L(X) \le 0$). By Lemma \ref{lmstrong}, $X$ can be represented as a wedge (one-point-union) $X=A\vee B$ with $A$ strongly connected and $B$ having $k-1$ strongly connected components. 
The complexes $A$ and $B$ are pure, closed, connected and $b_2(A)=0=b_2(B)$. 
One has
$L(X)=L(A)+L(B)$ and $\chi(X)=\chi(A)+\chi(B)-1$ which imply (using the notation introduced in the proof of Lemma \ref{btwo0})
$$\mathcal L(X) = \mathcal L(A) +\mathcal L(B)-2.$$
We know by the induction hypothesis that either $\mathcal L(A)\le 0$, or $A$ is homeomorphic to the projective plane and then $\mathcal L(A)=2$. 
Besides, either $\mathcal L(B) \le 0$ or $B$ is an iterated wedge of projective planes and then $\mathcal L(B)=2$. 
Thus, we see that either $\mathcal L(X) \le 0$ (which is equivalent to $\mu(X)\le 1/2$ and contradicts our assumption $\mu(X)>1/2$) or $\mathcal L(X)=2$ and $X$ is an iterated wedge of projective planes. 
\end{proof}

\begin{corollary} {\rm [Compare with Lemma 4.1 from \cite{BHK}]}\label{thm6} 
Let $Z$ be a connected 2-complex with $\tilde \mu(Z)>1/2$. Then $Z$ is homotopy equivalent to a wedge of circles, spheres and projective planes. 
\end{corollary}
\begin{proof} Without loss of generality we may additionally assume that $Z$ is strongly connected since otherwise $Z$ is homotopy equivalent to a wedge sum 
of its strongly connected components and a number of circles. 
Besides, we may assume that $Z$ is closed, $\partial Z=\emptyset$ (otherwise, we may perform a sequence of collapses to obtain a closed subcomplex $Z'\subset Z$ having the same homotopy type). Similarly, without loss of generality we may assume that $Z$ is pure since otherwise we can apply the arguments given below to its 
pure part. 

We shall act by induction on $b_2(Z)$. 

If $b_2(Z)=0$ and $\tilde \mu(Z)>1/2$ then using Lemma \ref{btwo01} we see that the complex $Z$ is homotopy equivalent to a wedge of circles and projective planes. 

Assume that Corollary \ref{thm6} was proven for all connected 2-complexes $Z$ satisfying $\tilde \mu(Z)>1/2$ and $b_2(Z)<k$. 

Consider a 2-complex $Z$ satisfying $b_2(Z)=k>0$ and $\tilde \mu(Z) >1/2$. 
Find a minimal cycle $Z'\subset Z$. Then the homomorphism $H_2(Z'; \Z)=\Z \to H_2(Z;\Z)$ is an injection. 
Let $\sigma\subset Z'$ be a simplex given by Corollary \ref{cor2}.
Then $Z''= Z- \int(\sigma)$ satisfies $b_2(Z'')=k-1$. Indeed, in the exact sequence
$$0\to H_2(Z'') \to H_2(Z) \to H_2(Z, Z'')\, (=\Z) \stackrel{\partial_\ast}\to H_1(Z'') \to \dots$$
the homomorphism $\partial_\ast=0$ is trivial since the curve $\partial \sigma$ is null-homotopic in $Z''$. 
Since $\tilde \mu(Z'') >1/2$, by induction hypothesis $Z''$ is homotopy equivalent to a wedge of spheres, circles and projective planes. 
Then $Z\simeq Z''\vee S^2$ also has this property. 
\end{proof}

\section{Isoperimetric constants of simplicial complexes}\label{sec22}
Let $X$ be a finite simplicial 2-complex. For a simplicial loop $\gamma: S^1\to X^{(1)}\subset X$ we denote by $|\gamma|$ the length of $\gamma$. If $\gamma$ is null-homotopic,
$\gamma \sim 1$, we denote by $A_X(\gamma)$ {\it the area} of $\gamma$, i.e. the minimal number of triangles in any simplicial filling $V$ for $\gamma$. 
A simplicial filling (or a simplicial Van Kampen diagram) for a loop $\gamma$ is defined as a pair of simplicial maps $S^1\stackrel{i}\to V\stackrel{b}\to X$ 
such that $\gamma=b\circ i$ and the mapping cylinder of $i$ is a disc with boundary $S^1\times 0$, see \cite{BHK}.   

Note that for $X\subset Y$ and $\gamma: S^1\to X$, $\gamma\sim 1$, one has 
$A_X(\gamma)\ge A_Y(\gamma).$

Define the following invariant of $X$
$$I(X)= \inf \left\{\frac{|\gamma|}{A_X(\gamma)}; \quad \gamma: S^1\to X^{(1)}, \gamma\sim 1 \quad \mbox{in $X$}\right\}\, \in \, \R.$$

The inequality $I(X)\ge a$ means that for any null-homotopic loop $\gamma$ in $X$ one has the isoperimetric inequality $A_X(\gamma)\le a^{-1}\cdot |\gamma|$. 
The inequality $I(X) < a$ means that there exists a null-homotopic loop $\gamma$ in $X$ with $A_X(\gamma) >a^{-1}\cdot |\gamma|$, i.e. $\gamma$ is null-homotopic but does not bound a disk of area less than $a^{-1}\cdot |\gamma|$. 

%\begin{remark}{\rm  Suppose that $\gamma=\gamma_1\cdot \gamma_2$ is the concatenation of two null-homotopic loops $\gamma_1$, $\gamma_2$ in $X$. Then 
%$$|\gamma|= |\gamma_1| +|\gamma_2|, \quad \mbox{and}\quad A_X(\gamma) \le A_X(\gamma_1) +A_X(\gamma_2).$$
%This implies that 
%$$\frac{|\gamma|}{A_X(\gamma)} \ge \frac{|\gamma_1|+|\gamma_2|}{A_X(\gamma_1) + A_X(\gamma_2)} \ge 
%\min\left\{\frac{|\gamma_1|}{A_X(\gamma_1)}, \frac{|\gamma_2|}{A_X(\gamma_2)}\right\}.$$
%}\end{remark}
%
%This leads to the following corollary:
%
%\begin{corollary} If $X$ is simply connected then %$I(X)$ coincides with the minimum of ratios $\frac{|\gamma|}{A_X(\gamma)}$ where $\gamma$ runs over all simple null-homotopic simplicial loops in $X$. Moreover, 
%there exists a simple (i.e. loop without self-intersections) simplicial null-homotopic loop $\gamma$ with $$\frac{|\gamma|}{A_X(\gamma)} =I(X).$$
%%$$I(X)= \inf \left\{\frac{|\gamma|}{A_X(\gamma)}; \quad \gamma: S^1\to X^{(1)}, \gamma\sim 1 \quad \mbox{in $X$}\right\}\, \in \, \R.$$
%\end{corollary}

It is well known that $I(X)>0$ if and only if $\pi_1(X)$ is {\it hyperbolic} in the sense of M. Gromov \cite{Gromov87}.

\begin{example}\label{extorus} {\rm For $X=T^2$ one has  $I(X) =0$. }
\end{example}

\begin{example} {\rm Let $X$ be a finite 2-complex satisfying $\tilde \mu(X)>1/2$. Then by Corollary \ref{thm6} the fundamental group $\pi_1(X)$ is a free product of several copies of $\Z$ and $\Z_2$ and is hyperbolic. Hence $I(X)>0$, compare Theorem \ref{uniform} below. }
\end{example}

\begin{remark}{\rm  Suppose that $\gamma=\gamma_1\cdot \gamma_2$ is the concatenation of two null-homotopic loops $\gamma_1$, $\gamma_2$ in $X$. Then 
$$|\gamma|= |\gamma_1| +|\gamma_2|, \quad \mbox{and}\quad A_X(\gamma) \le A_X(\gamma_1) +A_X(\gamma_2).$$
This implies that 
$$\frac{|\gamma|}{A_X(\gamma)} \ge \frac{|\gamma_1|+|\gamma_2|}{A_X(\gamma_1) + A_X(\gamma_2)} \ge 
\min\left\{\frac{|\gamma_1|}{A_X(\gamma_1)}, \frac{|\gamma_2|}{A_X(\gamma_2)}\right\}$$
and leads to the following observation:

{\it The number $I(X)$ coincides with the infimum of the ratios ${|\gamma|}\cdot{A_X(\gamma)}^{-1}$ where $\gamma$ runs over all null-homotopic simplicial prime loops in $X$, i.e. such that their lifts to the universal cover $\tilde X$ of $X$ are simple. 
Note that any simplicial filling $S^1\stackrel{i}\to V\stackrel{b}\to X$ for a prime loop $\gamma: S^1\to X$ has the property that 
$V$ is a simplicial disc and $i$ is a homeomorphism $i: S^1\to \partial V$.  Hence for prime loops $\gamma$ the area $A_X(\gamma)$ coincides with the minimal number of 2-simplexes in any simplicial spanning disc for $\gamma$. 
}

}\end{remark}

The following Theorem \ref{hyp} 
gives a uniform isoperimetric constant for random complexes $Y\in Y(n, p)$. It is a slightly stronger statement than simply 
hyperbolicity of the fundamental group of $Y$. Theorem \ref{hyp} is implicitly contained in \cite{BHK} although it was not stated there explicitly. 

\begin{theorem}\label{hyp} Suppose that for some $\epsilon>0$ the probability parameter $p$ satisfies
\begin{eqnarray}p\ll n^{-1/2-\epsilon}.\label{babsonrange}
\end{eqnarray}
Then there exists a constant $c_\epsilon>0$ depending only on $\epsilon$ such that a random 2-complex 
$Y\in Y(n, p)$, with probability tending to 1 as $n\to \infty$, has the following property: any subcomplex $Y'\subset Y$ 
satisfies $I(Y')\ge c_\epsilon$  and in particular, the fundamental group $\pi_1(Y')$ is hyperbolic. 
\end{theorem}

We give a detailed proof of Theorem \ref{hyp} in the Appendix. 

%\begin{example}
%{\rm Theorem \ref{hyp} implies that under the assumption (\ref{babsonrange}) a random 2-complex $Y\in Y(n,p)$, with probability tending to 1,
%has the following property: $Y$ contains no subcomplex 
%$S\subset Y$ satisfying  $\epsilon f(S)>|\partial S|$ which is homeomorphic to the disk $D^2$. 
%}
%\end{example}

We state now two immediate corollaries of Theorem \ref{hyp}.

\begin{corollary}\label{cortorus} Under the assumption (\ref{babsonrange}), a random 2-complex $Y\in Y(n,p)$ contains no subcomplexes 
homeomorphic to the torus $T^2$, a.a.s. 
\end{corollary}
\begin{proof} By Example \ref{extorus} any torus $S$ satisfies $I(S)=0$ and the inclusion $S\subset Y$ would contradict Theorem \ref{hyp}. 
\end{proof}

\begin{corollary}\label{corsphere} Under the assumption (\ref{babsonrange}), a random 2-complex $Y\in Y(n,p)$ contains no subcomplex 
having more than 
$2\epsilon^{-1}$ faces which is
homeomorphic to the sphere $S^2$, a.a.s. 
\end{corollary}
\begin{proof} Suppose that a random 2-complex $Y\in Y(n,p)$ contains a sphere $S\subset Y$ as a subcomplex. 

Firstly we show that the number of faces $f(S)$ is bounded above by a constant depending on $\epsilon$. Indeed, removing a 2-simplex from $S$ gives a disc 
$S'$ and by Theorem \ref{hyp} the isoperimetric ratio of the boundary curve $\gamma$ of this disc must satisfy 
$$\frac{|\gamma|}{A_{S'}(\gamma)} = \frac{3}{f(S)-1} \ge c_\epsilon$$
where $c_\epsilon>0$ is the constant of Theorem \ref{hyp}. Thus we obtain 
\begin{eqnarray}\label{uppersphere}f(S) \le 3\cdot c_\epsilon^{-1} +1.\end{eqnarray} 

Secondly, we observe that there are finitely many isomorphism types of triangulations of the sphere $S^2$ satisfying (\ref{uppersphere}) and we therefore may apply Theorem 15 from \cite{CCFK} or Lemma \ref{containment}. By Theorem 27 and formula (8) from \cite{CCFK} we obtain
$\tilde \mu(S) =\frac{1}{2} + \frac{2}{f(S)}.$
If $\tilde \mu(S) < 1/2 +\epsilon$ then (by Lemma \ref{containment}) $S$ cannot be embedded into a random 2-complex $Y\in Y(n,p)$, a.a.s. Hence, for the triangulated spheres $S$ embeddable into $Y$ 
(a.a.s.) we must have 
$\tilde \mu(S) = \frac{1}{2} + \frac{2}{f(S)} \ge \frac{1}{2} + \epsilon,$
implying $f(S) \le 2\cdot \epsilon^{-1}$. 
\end{proof}

%\begin{definition}
Let $X$ be a simplicial 2-complex. By $\rm {sys}(X)$ we denote {\it the systole} of $X$, which is defined as the length of the shortest homotopically nontrivial simplicial loop in $X$. %\end{definition}
By Theorem 1.3 from \cite{KRS} (see also Theorem 6.7.A from \cite{Gromov83}) one has 
\begin{eqnarray}\label{systole} {\rm {sys}}(X) < 6\cdot A(X)^{1/2},
\end{eqnarray}
assuming that the fundamental group of $X$ is not a free group.  Here $A(X)$ denotes the number of $2$-simplexes in $X$. 
To deduce (\ref{systole}) from  \cite{KRS}, Theorem 1.3 we apply this theorem to $X$ 
equipped with a piecewise flat metric $g$ in which each edge has length 1 and each 2-simplex has area $\sqrt{3}/4$. 
Let ${\rm {sys}}_g(X)$ and $A_g(X)$ denote the systole and the area of $X$ with respect to this metric. Then 
$A_g(X)=A(X)\cdot \sqrt{3}/4$ and elementary estimates show that
$${\rm {sys}}_g(X)\le {\rm {sys}}(X) \le 2\cdot{\rm {sys}}_g(X).$$
Hence, 
$$\frac{{\rm {sys}}^2(X)}{A(X)} \le \frac{\sqrt{3}\left({\rm {sys}}_g(X)\right)^2}{A_g(X)} \le \sqrt{3}\cdot 12 < 36$$
which implies (\ref{systole}).

\begin{corollary}\label{noprojplanes}
Under the assumption (\ref{babsonrange}), a random 2-complex $Y\in Y(n,p)$ contains no subcomplex
having more than 
$\epsilon^{-1}$ faces which is
homeomorphic to the projective plane ${\mathbf {RP}}^2$, a.a.s. 
\end{corollary}

\begin{proof} 
Suppose that a random 2-complex $Y\in Y(n,p)$ contains a subcomplex $P\subset Y$ homeomorphic to the projective plane.  
By Theorem \ref{hyp} the isoperimetric constant $I(P)$ satisfies $I(P)\ge c_\epsilon$. Consider the simple simplicial curve $\gamma$ in $P$ which is not null-homotopic and has length equal to 
$\sys(P)$. The twice passed curve $\gamma$ is null-homotopic in $P$ and the area of $\gamma^2$ equals $A(P)$. Thus we have 
\begin{eqnarray}\label{syst1}
\frac{2\sys(P)}{A(P)} \ge c_\epsilon.
\end{eqnarray}
Applying (\ref{systole}) 
 we obtain $\sys(P)^2 \le 36\cdot A(P)$ and combining with (\ref{syst1}) we get $A(P) \le 144\cdot c_\epsilon^{-2}$. 

There are finitely many isomorphism types of triangulations of the real projective plane ${\mathbf {RP}}^2$ with at most $144 c_\epsilon^{-2}$ faces. %and therefore  we may apply Theorem 15 from \cite{CCFK}. 
By Theorem 27 and formula (8) from \cite{CCFK} we obtain
$\tilde \mu(S) =\frac{1}{2} + \frac{1}{f(S)}.$
If $\tilde \mu(S) < 1/2 +\epsilon$ then (by Lemma \ref{containment}) $S$ cannot be embedded into a random 2-complex $Y\in Y(n,p)$, a.a.s. Hence, for the triangulated real projective plane $S$ embeddable into $Y$ 
(a.a.s.) we must have 
$\tilde \mu(S) = \frac{1}{2} + \frac{1}{f(S)} \ge \frac{1}{2} + \epsilon,$
implying $f(S) \le \epsilon^{-1}$. 
\end{proof}

\section{Torsion in fundamental groups of random 2-complexes}

One of the main results of this section is Theorem \ref{ordertwo} which states that the fundamental group of a random 2-complex has a nontrivial element of order $2$ assuming that $n^{-3/5} \ll p\ll n^{-1/2-\epsilon}$. The proof of Theorem \ref{ordertwo} uses Theorem \ref{hyp} together with Theorem \ref{family} which is stated and proven below. 

The second main result presented in this section is Theorem \ref{orderthree} which claims that for any odd prime $m\ge 3$ (assuming that $p\ll n^{-1/2-\epsilon}$)
the fundamental group of any subcomplex 
$Y'\subset Y$ of a random complex $Y\in Y(n,p)$  has no $m$-torsion, a.a.s. The proof uses Theorem \ref{hyp} as well as the inequalities for systoles of Moore surfaces.

\subsection{The numbers of embeddings}

Consider two 2-complexes $S_1\supset S_2$. Denote by $v_i$ and $f_i$ the numbers of vertices and faces of $S_i$. We have $v_1\ge v_2$ and $f_1\ge f_2$.
We will assume that $f_1>f_2$. 

Let $\mu(S_1, S_2)$ denote the ratio $$\mu(S_1, S_2)= \frac{v_1-v_2}{f_1-f_2}.$$

If $\mu(S_1)<\mu(S_2)$ then 
\begin{eqnarray}\label{one}\mu(S_1, S_2) <\mu(S_1)<\mu(S_2).\end{eqnarray}
If $\mu(S_1) >\mu(S_2)$ then 
\begin{eqnarray}\label{two}\mu(S_1, S_2) >\mu(S_1)>\mu(S_2).\end{eqnarray}
These two observations can be summarised by saying that {\it $\mu(S_1)$ always lies in the interval connecting $\mu(S_2)$ and $\mu(S_1, S_2)$. }

One has the following formula
\begin{eqnarray}\label{24}
\mu(S_1, S_2) = \frac{1}{2} +\frac{2(\chi(S_1)-\chi(S_2)) + L(S_1)-L(S_2)}{2(f_1-f_2)},
\end{eqnarray}
which follows from the equation $2v_i = f_i + 2\chi(S_i) +L(S_i)$; the latter is equivalent to (\ref{mu}). 

\begin{remark} {\rm Let $S_2$ be a pseudo-surface and
$2\chi(S_1,S_2)+L(S_1) <0$. Then $\mu(S_1, S_2) <1/2$. }
\end{remark}
\noindent
Note that $L(S_2)=0$ since $S_2$ is a pseudo-surface. This remark applies to the case when $S_1$ is a union of a pseudo-surface $S_2$ and a number of simplicial discs.

\begin{theorem}\label{thm1} Let $S_1\supset S_2$ be two fixed 2-complexes 
and\footnote{The assumption (\ref{between}) is meaningful iff $\mu(S_1, S_2) < \tilde \mu(S_2) \le \mu(S_2)$ which, as follows from (\ref{one}) and (\ref{two}), 
implies that $\mu(S_1)<\mu(S_2)$. Thus, if Theorem \ref{thm1} is applicable, then $\mu(S_1)<\mu(S_2)$. }
\begin{eqnarray}\label{between}
n^{-\tilde \mu(S_2)} \ll p \ll n^{-\mu(S_1, S_2)}.
\end{eqnarray}
Then the number of simplicial embeddings of $S_1$ into a random 2-complex $Y\in Y(n, p)$ is smaller than the number of simplicial embeddings of $S_2$ into $Y$, a.a.s. 
In particular, under the assumptions (\ref{between}), with probability tending to one, there exists a simplicial embedding $S_2\to Y$ which does not extend to an embedding $S_1\to Y$. 
\end{theorem}
\begin{proof}
Let $X_i: Y(n, p)\to \Z$ be the random variable counting the number of simplicial embeddings of $S_i$ into $Y\in Y(n, p)$, $i=1,2$. We know that 
$$\E(X_i)=\binom n {v_i} v_i! p^{f_i} \sim n^{v_i}p^{f_i}.$$ 
Our goal is to show that $X_1<X_2$, a.a.s. The left hand side inequality (\ref{between}) implies (via Theorem 15 from \cite{CCFK}) that $X_2(Y)>0$, a.a.s., i.e. 
$S_2$ admits an embedding into $Y$ with probability tending to one. 
We have 
$$\frac{\E(X_1)}{\E(X_2)} \sim n^{v_1-v_2}p^{f_1-f_2} = \left[ n^{\mu(S_1,S_2)}p\right]^{f_1-f_2}\to 0$$
tends to zero, under our assumption (\ref{between}) (the right hand side). 

Below we shall find $t_1, t_2 >0$ such that 
$$t_1+t_2=\E(X_2)-\E(X_1)$$ and $\E(X_1)/t_1\to 0$ while $\E(X_2)/t_2$ is bounded. 

One of the following three statements holds: either $X_1<X_2$ or $ X_1\ge \E(X_1) +t_1$ or $X_2<\E(X_2) -t_2$ and therefore
\begin{eqnarray}\label{new3}
P(X_1<X_2) \ge 1- P(X_1\ge \E(X_1) +t_1) - P(X_2< \E(X_2) -t_2).\end{eqnarray}
By Markov's inequality 
$$P(X_1\ge\E(X_1)+t_1) \le \frac{\E(X_1)}{\E(X_1)+t_1}= \frac{\frac{\E(X_1)}{t_1}}{1+ \frac{\E(X_1)}{t_1}}\to 0$$
while by Chebyschev's inequality
$$P(X_2< \E(X_2)-t_2) < \frac{{\rm {Var(X_2)}}}{t_2^2}.$$
It is known (see \cite{CCFK}, proof of Theorem 15) that under our assumptions (\ref{between}) the ratio 
$\frac{{\rm {Var(X_2)}}}{\E(X_2)^2}$ tends to zero. Therefore combining the last two inequalities with the inequality (\ref{new3}) we see that 
$P(X_1<X_2)$ tends to 1 as $n\to \infty$.

To make a specific choice of $t_1$ and $t_2$ one may take $t_1=\sqrt{\E(X_1)\E(X_2)}$ and $t_2= \E(X_2) - \E(X_1) - t_1$. Then 
$$\frac{\E(X_1)}{t_1} = \sqrt{\frac{\E(X_1)}{\E(X_2)}} \to 0$$
and 
$$\frac{\E(X_2)}{t_2} = \frac{1}{1-\frac{\E(X_1)}{\E(X_2)}- \sqrt{\frac{\E(X_1)}{\E(X_2)}}}\to 1$$
is bounded. This completes the proof. 
\end{proof}

\begin{theorem}\label{family} Let $$T_{ j}\supset S, \quad j=1, \dots, N,$$ be a finite family of 2-complexes containing a given 2-complex $S$ and satisfying  $\mu(T_{ j})<\mu(S)$. Assume that 
\begin{eqnarray}\label{between2}
n^{-\tilde \mu(S)} \ll p \ll n^{-\mu(T_{j}, S)}, \quad \mbox{for any} \quad j=1, \dots, N.
\end{eqnarray}
Then,
with probability tending to one, for a random 2-complex $Y\in Y(n, p)$ there exists a simplicial embedding $S\to Y$ which does not extend to a simplicial embedding $T_{j}\to Y$, for any $j=1, \dots, N$. 
\end{theorem}
\begin{proof} Let $X_{1, j}: Y(n, p) \to \Z$ denote the random variable counting the number of embeddings of $T_j$ into $Y\in Y(n, p)$. Denote 
$X_1= \sum_{j=1}^N X_{1, j}.$ Besides, let $X_2: Y(n, p)\to \Z$ denote the number of embeddings of $S$ into a random 2-complex $Y\in Y(n, p)$. 
As in the proof of the provious theorem one has 
$$\frac{\E(X_1)}{\E(X_2)} = \sum_{j=1}^N \frac{\E(X_{1, j})}{\E(X_2)} \to 0$$
thanks to our assumption (\ref{between2}) (the right hand side inequality). Taking 
$t_1= \sqrt{\E(X_1)\E(X_2)}$ and $t_2=\E(X_2)-\E(X_1)-t_1$ (as in the proof of the previous theorem) one has 
$t_1+t_2=\E(X_2) - \E(X_1)$ and $\E(X_1)/t_1\to 0$ while $\E(X_2)/t_2$ is bounded. Repeating the arguments used in the proof of the previous theorem we see that 
$X_1<X_2$, a.a.s. Since every embedding $T_j\to Y$ determines (by restriction) an embedding $S\to Y$, the inequality $X_1(Y)<X_2(Y)$ implies that there there are embeddings $S\to Y$ which admit no extensions to an embedding $T_j\to Y$, for any $j=1, \dots, N$. 
\end{proof}

\subsection{Projective planes in random 2-complexes}

In this section we prove the existence of 2-torsion in fundamental groups of random 2-complexes, see Theorem \ref{ordertwo}. 
The proof uses simplical embeddings of projective planes into random 2-complexes. 
Note that for a triangulation $X$ of the real projective plane one has 
$$\mu(X) = 1/2+\frac{1}{f(X)},$$ 
where $f(X)$ is the number of faces of $X$, see  (\ref{mu}), and therefore the maximal value of $\mu(X)$ happens when $f(X)$ is minimal. 

It is well known that the complex $S$
shown in the Figure \ref{figproj} 
 is the minimal triangulation of $P^2$; it has 10 faces and therefore $\mu(S)=3/5$. Moreover $\tilde \mu(S) =\mu(S)=  3/5$ by Theorem 27 from \cite{CCFK}.
By Lemma \ref{containment} for $p\gg n^{-3/5}$ the complex $S$ embeds into a random 2-complex, a.a.s. This explains the appearance of the exponent 
$3/5$ in Theorem \ref{ordertwo} below. \begin{figure}[h]
\centering
\includegraphics[width=0.3\textwidth]{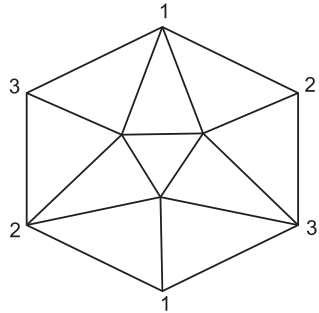}
\caption{Minimal triangulation of the real projective plane; the antipodal points on the peripheral hexagon must be identified.}\label{figproj}
\end{figure}

We may mention here that for $p\ll n^{-3/5}$ the fundamental group of a random 2-complex $Y\in Y(n,p)$ is torsion free, see Corollary \ref{cor35}.

\begin{definition} A subcomplex $S\subset Y$ is said to be essential if the induced homomorphism $\pi_1(S) \to \pi_1(Y)$ is injective. 
\end{definition}

\begin{theorem}\label{ordertwo} Let $S$ be a triangulation of the real projective plane ${\mathbf {RP}}^2$ having 6 vertices, 15 edges and 10 faces. Assume that $\epsilon >0$ and 
$$n^{-3/5}\ll p\ll n^{-1/2 - \epsilon}.$$ Then a random 2-complex $Y\in Y(n, p)$ contains $S$ as an essential subcomplex, a.a.s. In particular, the fundamental group 
$\pi_1(Y)$ contains an element of order two and hence its cohomological dimension is infinite. 
\end{theorem}
\begin{proof} 
Consider the set $\mathcal S_\epsilon$ of isomorphism types of pure connected closed 2-complexes $X$ satisfying the following conditions:
\begin{enumerate}
  \item[(a)] $X$ contains $S$ as a subcomplex; 
  \item[(b)] The inclusion $S\to X$ induces a trivial homomorphism $\pi_1(S) \to \pi_1(X)$; 
  \item[(c)] For any subcomplex $S\subset X'\subset X$, $X'\not=X$, the homomorphism $\pi_1(S)\to \pi_1(X')$ is nontrivial;
  \item[(d)]  $X$ has at most $10+3c_\epsilon^{-1}$ faces where $c_\epsilon$ is the constant given by Theorem \ref{hyp};
  \item[(e)] $\tilde \mu(X) > 1/2 +\epsilon$;
  \end{enumerate}

Let us show that any $X\in {\mathcal S}_\epsilon$ is homeomorphic to the complex $Z_4$, as defined in Theorem \ref{lm4}.
In the exact sequence 
$$0\to H_2(X) \to H_2(X,S) \to H_1(S)=\Z_2 \to 0$$
the middle group has no torsion and hence $b_2(X)=b_2(X,S) \ge 1$. If $Z\subset X$ is a minimal cycle then $\mu(Z)>1/2+ \epsilon$ (by (e)) and Theorem \ref{lm4} implies that $Z$ is homeomorphic to one of the 2-complexes $Z_1, Z_2, Z_3, Z_4$. 

If $Z$ is homeomorphic to one of the complexes $Z_1, Z_2, Z_3$ then any face $\sigma \subset Z$ satisfies the conditions of Corollary \ref{cor2}. Therefore 
removing from $X$
any face $\sigma\subset Z-S$ would produce 2-complex $X'$ violating (c). 

This shows that any minimal cycle $Z\subset X$ is homeomorphic to $Z_4=P^2\cup \Delta^2$. 

The union of the edges of degree 3 in $Z$ is a closed curve and cutting along this curve disconnects 
$Z$ onto a copy of $P^2$ and $\Delta^2$. Thus, we may think of $P^2$ as being a subcomplex of $Z$. If $P^2\not\subset S$ then 
there is a 2-simplex $\sigma\subset P^2-S$. Then the curve $\partial \sigma\subset Z-\sigma\subset X-\sigma$ is null-homotopic and hence 
the inclusion $X'=Z-\sigma \to X$ induces an isomorphism $\pi_1(X') \to \pi_1(X)$ contradicting (c). 
This argument shows that $P^2 \subset S $ and hence $P^2=S$. 
Since $\pi_1(S) \to \pi_1(Z)$ is trivial, it follows from the minimality of $X$ (property (c)) that $Z=X$. 

Since $\mu(X)>1/2$, by formula (\ref{mu}) we obtain that $2\chi(X)+L(X)>0$ or equivalently $L(X)\ge -3$ (as $\chi(X)=2$ since $X$ is homotopy equivalent to the 2-sphere). On the other hand, $L(X)\le -3$ 
since $X$ is closed and has at least 3 edge of degree 3.  Thus $L(X)=-3$.

Using formula (\ref{24}) we find that 
$$\mu(X,S) = \frac{1}{2} - \frac{1}{2(f(X)-10)}< \frac{1}{2}$$
and then applying Theorem \ref{family} we find that for 
$$n^{-3/5}\ll p \ll n^{-1/2-\epsilon},$$ with probability tending to 1, there exist embedding $S\to Y$ (where $Y\in Y(n,p)$ is random) which cannot be extended to an embedding of $X\to Y$ for any $X\in {\mathcal S}_\epsilon$.

Let $Y'_n\subset Y(n,p)$ denote the set of complexes $Y\in Y(n,p)$ such that there exists an simplicial embedding $S\to Y$ which cannot be extended to an embedding $X\to Y$, for any $X\in {\mathcal S}_\epsilon$. We have shown earlier that ${\mathbb P}(Y_n')\to 1$ as $n\to \infty$. 

Consider the finite family $\mathcal F$ of isomorphism types of simplicial complexes $T$ having at most $10+3c_\epsilon^{-1}$ faces and satisfying 
$\tilde \mu(T) \le 1/2 + \epsilon$. By Lemma \ref{containment}, since we assume that $p\ll n^{-1/2-\epsilon}$, 
the set $Y''_n\subset Y(n,p)$ of 2-complexes $Y\in Y(n,p)$ having the property that none of the complexes $T\in {\mathcal F}$ is embeddable into $Y$, satisfies 
${\mathbb P}(Y''_n)\to 1$, as $n\to \infty$. 

Besides, let $Y'''_n\subset Y(n,p)$ denote the set of complexes satisfying Theorem \ref{hyp}. We know that ${\mathbb P}(Y'''_n)\to 1$ as $n\to \infty$.

Now, let $Y\in Y_n'\cap Y_n''\cap Y'''_n$ and let $S\subset Y$ be an embedding which cannot be extended to an embedding $X\to Y$ for 
$X\in {\mathcal S}_\epsilon$. Let us show that $S\subset Y$ is essential. 
%
%Theorem 15 from \cite{CCFK} implies that under the assumptions $p\gg  n^{-3/5}$ the complex $S$ is embeddable into a random 2-complex 
%$Y\in Y(n,p)$ a.a.s. 
Assuming the contrary, if the embedding $S\subset Y$ is not essential then the central circle $\gamma$ of $S$ (of length 3) extends to a simplicial map of a simplicial disc into $Y$. 
Under the assumption $p\ll n^{-1/2-\epsilon}$, using Theorem \ref{hyp}, we find that the circle $\gamma$ extends to a simplicial map 
$b: \Delta^2 \to Y$, $b|\partial \Delta^2=\gamma$, 
into $Y$ of a simplicial disc  
of area $\le 3c_\epsilon^{-1}$ where $c_\epsilon>0$ depends only on the value of $\epsilon$. 
Since $S\cup b(\Delta^2)$ is embedded into $Y$ we obtain that $\tilde \mu (S\cup b(\Delta^2))>1/2 +\epsilon$. 
Hence, if $S\subset Y$ is not essential then there would exist a complex $X\in {\mathcal S}_\epsilon$, 
where $S\subset X\subset S\cup b(\Delta^2)$, such that the embedding $S\subset Y$ extends to an embedding $X\subset Y$. This gives a contradiction. 
\end{proof}

\subsection{Absence of higher torsion}

In this subsection we prove the following statement complementing Theorem \ref{ordertwo}. 

\begin{theorem}\label{orderthree} Assume that the probability parameter $p$ satisfies
\begin{eqnarray}
p\ll n^{-1/2-\epsilon}, 
\end{eqnarray} 
where
$\epsilon>0.$ Let $m\ge 3$ be a prime number. Then a random 2-complex $Y\in Y(n,p)$, with probability tending to 1, has the following property: 
for any subcomplex $Y'\subset Y$ the fundamental 
group $\pi_1(Y')$ has no elements of order $m$.
\end{theorem}

%We say that an element $g\in G$ of a discrete group  has order $m$ if the subgroup generated by $g$ is isomorphic to $\Z_m$. 

Let $\Sigma$ be a simplicial 2-complex homeomorphic to {\it the Moore surface} 
$$M(\Z_m, 1)=S^1\cup_{f _m}e^2, \quad \mbox{where}\quad \quad m\ge 3;$$ 
it is obtained from the circle $S^1$ by attaching a 2-cell via the degree $m$ map $f_m: S^1\to S^1$, $f_m(z)=z^m$, $z\in S^1$. 
The 2-complex $\Sigma$ has a well defined circle $C\subset \Sigma$ (which we shall call {\it the singular circle}) which is the union of all edges of degree $m$; all other edges of $\Sigma$ have degree $2$. Clearly, the homotopy class of the singular circle generates the fundamental group $\pi_1(\Sigma)\simeq \Z_m$. 

Next we define an integer $N_m(Y)\ge 0$ associated to any connected 2-complex $Y$. If $\pi_1(Y)$ has no $m$-torsion we set $N_m(Y)=0.$
If $\pi_1(Y)$ has elements of order $m$ we shall consider homotopically nontrivial simplicial maps 
$\gamma: C_r \to Y$,
where $C_r$ is the simplicial circle with $r$ edges, such that
\begin{enumerate}
  \item[(a)] $\gamma^m$ is null-homotopic (as a free loop in $Y$);
  %\item[(b)] for any $1\le j<m$ the loop $\gamma^j$ is not null-homotopic in $Y$;
  \item[(b)] $r$ is minimal: for $r'<r$ any simplicial loop $\gamma:C_{r'} \to Y$ satisfying (a) is homotopically trivial. 
  \end{enumerate} 
Any such simplicial map $\gamma:C_r \to Y$ can be extended to a simplicial map $f: \Sigma \to Y$ of a triangulation $\Sigma$ of the Moore surface, such that the singular circle $C$ of $\Sigma$ is isomorphic to $C_r$ and $f|C=\gamma$. We shall say that a simplicial map $f:\Sigma \to Y$ is {\it $m$-minimal} if 
it satisfies (a), (b) and the number of 2-simplexes in $\Sigma$ is the smallest possible. 
Now, we denote by 
$$N_m(Y)\in \Z$$ 
the number of 2-simplexes in a triangulation of the Moore surface $\Sigma$ admitting an $m$-minimal map $f: \Sigma \to Y$.

%\begin{lemma}\label{lmb}
%Let $Y$ be a 2-complex with $I(Y)\ge c>0$. Let $f: \Sigma \to Y$ be an essential simplicial map, where $\Sigma$ is a triangulation of $M(\Z_m, 1)$ and 
%$A(\Sigma) = N_m(Y)$. Then $${\rm  {sys}}(\Sigma) \ge \frac{c}{m}\cdot A(\Sigma).$$ 
%\end{lemma}

\begin{lemma}\label{lmc}
Let $Y$ be a 2-complex satisfying $I(Y)\ge c>0$ (where the quantity $I(Y)$ is defined in \S \ref{sec22}). 
Let $m\ge 3$ be an odd prime. 
Then
%for some universal constant\footnote{We know from \cite{KRS} that $B\le 3^{3/4}\sim 2.2795$} $B$, 
one has 
$$N_m(Y) \le \left(\frac{6m}{c}\right)^2. $$
\end{lemma}

\begin{proof} We shall assume that the fundamental group of $Y$ contains an element of order $m$; otherwise $N_m(Y)=0$. 
Consider an $m$-minimal simplicial map $f: \Sigma \to Y$,  
\begin{eqnarray}\label{sys1}
A(\Sigma) = N_m(Y),
\end{eqnarray}
where $A(\Sigma)$ denotes the number of 2-simplexes in $\Sigma$. 
It is obvious from the $m$-minimality of $f$ that the singular circle $C\subset \Sigma$ is the shortest (in terms of the number of edges) homotopically nontrivial simplicial loop 
in $\Sigma$, i.e.
\begin{eqnarray}\label{shortest}
|C|={\rm {sys}}(\Sigma).\end{eqnarray}

Consider the loop $\gamma=f|C$, $|\gamma|=|C|={\rm {sys}}(\Sigma)$. We know that $\gamma^m$ is homotopically trivial in $Y$ and the inequality $I(Y) \ge c>0$ implies that $\gamma^m$ bounds in $Y$ a disc of area at most 
$m\cdot |\gamma| \cdot c^{-1}$. The 2-complex 
$\Sigma$ is obtained from a simplicial disc $D$ by dividing its boundary into $m$ intervals of equal length and identifying them to each other. 
If $A(D) > m\cdot |\gamma| \cdot c^{-1}$ then one could replace the disc $D$ by the minimal spanning disc for $\gamma^m$ obtaining a simplicial map $\Sigma'\to Y$ which has smaller area contradicting the $m$-minimality. Thus, we obtain
\begin{eqnarray}\label{syss2}
A(\Sigma) \le c^{-1}\cdot m \cdot {\rm {sys}}(\Sigma).
\end{eqnarray}
Inequality (\ref{systole}) gives
%
%Theorem 6.7.A from \cite{Gromov83} and by Theorem 1.3 from \cite{KRS}, there exists a universal constant $B>0$ such that\footnote{To apply these theorems we 
%equip $\Sigma$ with a piecewise flat metric in which each edge has length 1 and each 2-simplex has area $\sqrt{3}/4$.}
%
\begin{eqnarray}\label{sys3} {\rm {sys}}(\Sigma) \le 6\cdot A(\Sigma)^{1/2}. 
\end{eqnarray}
Combining (\ref{sys1}), (\ref{syss2}) and (\ref{sys3}) we obtain 
$$N_m(Y)\,  =\,  A(\Sigma) \le \left(\frac{6m}{c}\right)^2. $$
\end{proof}
\begin{theorem}\label{thma}
Assume that the probability parameter $p$ satisfies $p\ll n^{-1/2-\epsilon}$ where $\epsilon >0$ is fixed. 
Let $m\ge 3$ be an odd prime. 
Then there exists a constant $C_\epsilon>0$ such that a random 2-complex $Y\in Y(n,p)$ with probability tending to 1 has the following property: for any subcomplex $Y'\subset Y$ one has
\begin{eqnarray}\label{ineqq}
N_m(Y') \le C_\epsilon. 
\end{eqnarray}
\end{theorem}
\begin{proof}  We know from Theorem \ref{hyp} that, with probability tending to 1, a random 2-complex $Y$ has the following property: for any subcomplex $Y'\subset Y$ one has 
$I(Y')\ge c_\epsilon>0$ where $c_\epsilon>0$ is the constant given by Theorem \ref{hyp}. 
Then, setting $C= \left(\frac{6m}{c_\epsilon}\right)^2$, the inequality (\ref{ineqq}) follows from Lemma \ref{lmc}. 
\end{proof}

\begin{proof}[Proof of Theorem \ref{orderthree}] Let $c_\epsilon>0$ be the number given by Theorem \ref{hyp}. 
Consider the finite set of all isomorphism types of triangulations $\mathcal S_m =\{\Sigma\}$ of the Moore surface $M(\Z_m,1)$ having at most 
$\left(\frac{6m}{c_\epsilon}\right)^2$ two-dimensional simplexes. Let $\mathcal X_m$ denote 
the set of isomorphism types of images of all surjective simplicial maps $\Sigma \to X$ inducing injective homomorphisms $\pi_1(\Sigma)=\Z_m \to \pi_1(X)$,
where $\Sigma \in \mathcal S_m$. 
The set $\mathcal X_m$ is also finite.

From Theorem \ref{thma} we obtain that, with probability tending to one, for any subcomplex $Y'\subset Y$, either $\pi_1(Y')$ has no $m$-torsion, or 
there exists an $m$-minimal map $f: \Sigma \to Y'$ with $\Sigma$ having at most $\left(\frac{6m}{c_\epsilon}\right)$ simplexes of dimension 2; in the second case the image 
$X=f(\Sigma)$ is a subcomplex of $Y'$ and $f:\Sigma \to X$ induces a monomorphism $\pi_1(\Sigma)\to \pi_1(X)$, i.e. 
$X\in {\mathcal X}_m$.

From Corollary \ref{thm6} we know that the fundamental group of any 2-complex satisfying $\tilde \mu(X)>1/2$ is a free product of several copies of $\Z$ and $\Z_2$ and has no $m$-torsion, as we assume that $m\ge 3$. 
Since the fundamental group of any $X\in \mathcal X_m$ has $m$-torsion, where $m\ge 3$, one has $\tilde \mu(X) \le 1/2$ for any $X\in \mathcal X_m$. 
 Hence, using the finiteness of $\mathcal X_m$ and the results on the containment problem (see \cite{CCFK}, Theorem 15) we see that for 
$p\ll n^{-1/2-\epsilon}$
the probability that a random complex $Y\in Y(n,p)$ contains a subcomplex isomorphic to one of the complexes
$X\in \mathcal X_m$ tends to $0$ as $n\to \infty$. 
Hence, we obtain that (a.a.s.) any subcomplex $Y'\subset Y$ does not contain $X\in \mathcal X_m$ as a subcomplex and therefore the fundamental group of 
$Y'$ has no $m$-torsion.
\end{proof}

\section{Minimal spheres and the Whitehead Conjecture}\label{secwh}

\begin{definition}
Let $Y$ be a simplicial complex with $\pi_2(Y)\not=0$. We define a numerical invariant $M(Y)\in \Z$, $M(Y) \ge 4$, as the minimal number of faces in a 
2-complex $\Sigma$ 
homeomorphic to the sphere
$S^2$ such that there exists a homotopically nontrivial simplicial map $\Sigma\to Y$. 
We define $M(Y)=0$, if $\pi_2(Y)=0$.
\end{definition}

Our first goal in this section is to prove the following theorem:

\begin{theorem} \label{thmupper}
Assume that the probability parameter $p$ satisfies $p\ll n^{-1/2-\epsilon}$ where $\epsilon >0$ is fixed. Then there exists a constant $C=C_\epsilon$ 
such that 
a random 2-complex $Y\in Y(n,p)$ with probability tending to 1 has the following property: for any subcomplex $Y'\subset Y$ one has 
\begin{eqnarray}\label{mless}
M(Y') \le C.  
\end{eqnarray}
\end{theorem}

In other words, in the specified range of $p$, the numbers $M(Y')$ have a uniform upper bound, a.a.s.

In the proof of Theorem \ref{thmupper} we will use the following version of the notion of Cheeger's constant. 
For a pure simplicial 2-complex $\Sigma$ one defines the Cheeger's constant $h(\Sigma)$ as 
\begin{eqnarray}\label{chee}h(\Sigma) = \min_{S\subset \Sigma}\left\{\frac{|\partial S|}{A(S)}; A(S)\le A(\Sigma)/2\right\},\end{eqnarray}
where $S\subset \Sigma$ runs over all pure subcomplexes. % with one boundary component $\partial S=\overline{(\Sigma -S)}\cap S$. 
Here $|\partial S|$ denotes the length of the boundary (the number of edges) and $A(S)$ and $A(\Sigma)$ denote the area of $S$ and $\Sigma$, i.e. the number of 2-simplexes. 

Note that in (\ref{chee}) one may always assume that $S\subset \Sigma$ is strongly connected. Indeed, if $S=\cup_j S_j$ are the strongly connected components of $S$ then $|\partial S|= \sum_j |\partial S_j|$ and $A(S) = \sum_j A(S_j)$ and therefore,
$$\frac{|\partial S|}{A(S)} \ge \min_j \frac{|\partial S_j|}{A(S_j)}.$$

In the case when $\Sigma$ is homeomorphic to the sphere $S^2$ one can require the subcomplex $S\subset \Sigma$ which appears in the formula (\ref{chee}) to be homeomorphic to a disc. Indeed, if $S\subset \Sigma$ is strongly connected and $\frac{|\partial S|}{A(S)}=h(S)$, $A(S) \le A(\Sigma)/2$, consider the strongly connected components  of 
the closure of the complement $\overline{\Sigma-S}=\cup_j D_j$. Each $D_j$ is a disc and $|\partial S|=\sum_j |\partial D_j|$ and 
$A(S) \le A(\Sigma - S)=\sum_j A(D_j)$. Hence
$$h(\Sigma) = \frac{|\partial S|}{A(S)} \ge \frac{\sum_j|\partial D_j|}{\sum_j A(D_j)}\ge \min_j \frac{|\partial D_j|}{A(D_j)}$$
and in particular for some $j_0$ one has $\frac{|\partial D_{j_0}|}{A(D_{j_0})}\le h(\Sigma)$. From now on we may assume that $A(D_{j_0})>A(\Sigma)/2$ since otherwise the proof is complete. Consider the disc $C=\overline{\Sigma-D_{j_0}}$. Then $A(C) \le A(\Sigma)/2$ and
$$A(C) \ge A(S), \quad |\partial C|\le |\partial S|$$ and we obtain 
%$A(C)<A(\Sigma)/2<A(D_{j_0})$ and $|\partial C|=|\partial D_{j_0}|$. Therefore, we obtain 
$$h(\Sigma) = \frac{|\partial S|}{A(S)} \ge \frac{|\partial C|}{A(C)}.$$
Thus the value $h(\Sigma)$ in (\ref{chee}) is achived on subdiscs. 

Our next two results are deterministic. We show that the isoperimetric constant can be used to estimate above the value $M(Y)$.  

\begin{lemma}\label{lmcheeger}
Let $Y$ be a 2-complex with $I(Y)\ge c >0$. Let $f: \Sigma\to Y$ be a homotopically nontrivial simplicial map where $\Sigma$ is homeomorphic to $S^2$ and 
$A(\Sigma) = M(Y)$. Then $h(\Sigma)\ge c$. 
\end{lemma}
\begin{proof} Assume that $h(\Sigma)<c$. Then (due to our discussion above) there exists a simplicial subdisc $S\subset \Sigma$ with $A(S) \le A(\Sigma)/2$ and $\frac{|\partial S|}{A(S)}<c$. 
The curve $\gamma=f|\partial S$ is a null-homotopic loop in $Y$ and $f|S: S\to Y$ is a bounding disk for $\gamma$. We claim that this disc has the smallest area among all bounding discs for $\gamma$ in $Y$. Indeed, if there existed a spanning disc $b: D\to Y$, $b|\partial D=\gamma$, with $A(D)<A(S)$ then we may define two 
maps of the sphere $S^2\to Y$, both having smaller area than $f$.
These maps are $\phi: S\cup D \to Y$ and 
$\phi': S'\cup D\to Y$ where $S'=\overline{\Sigma-S}$. Here $\phi|S = f|S$, $\phi|D=b=\phi'|D$, $\phi'|S'=f|S'$. 
Clearly at least one of $\phi$ and $\phi'$ is 
not null-homotopic, since if both $\phi$ and $\phi'$ were null-homotopic then the original map $S^2 \to Y$ was null-homotopic as well. 
Thus we obtain 
$A(S)=A_Y(\gamma)$ and 
$\frac{|\gamma|}{A_Y(\gamma)}<c$
contradicting our assumption $I(Y)\ge c$. 
\end{proof}

\begin{corollary}\label{cory}
Let $Y$ be a 2-complex with $I(Y)\ge c >0$. Then 
\begin{eqnarray} 
M(Y) \le \left(\frac{16}{c}\right)^2.
\end{eqnarray}
\end{corollary}
\begin{proof} Papasoglu \cite{Pap} proved the inequality $$h(\Sigma) \le \frac{16}{\sqrt{A(\Sigma)}}$$ valid for any simplicial sphere. By Lemma \ref{lmcheeger}, one 
has  
$h(\Sigma)>c$ for the 
homotopically nontrivial map $\Sigma \to Y$ with the minimal $A(\Sigma)$. This implies that $A(\Sigma) \le 16^2/c^2$ which is equivalent to our statement. 
\end{proof}

\begin{proof}[Proof of Theorem \ref{thmupper}] By Theorem \ref{hyp} one has $I(Y') > c_\epsilon$, a.a.s. where $c_\epsilon>0$ depends only on $\epsilon$. 
Thus, in view of Corollary \ref{cory}, the inequality (\ref{mless}) is satisfied with $C=256/c^2_\epsilon$. 
\end{proof}

%\section{The Whitehead Conjecture}

Next we characterise aspherical subcomplexes of random 2-complexes:

\begin{theorem}\label{thmasph} Assume that $p\ll n^{-1/2-\epsilon}$ for a fixed $\epsilon >0$. 
Then a random 2-complex $Y\in Y(n,p)$ has the following property with probability tending to one as $n\to \infty$: 
for any subcomplex $Y'\subset Y$ the following properties are equivalent:
 \begin{description}
  \item[(A)] $Y'$ is aspherical;
  \item[(B)] $Y'$ contains no subcomplexes $S\subset Y'$ with at most $2\epsilon^{-1}$ faces which are homeomorphic to the sphere $S^2$, the projective plane ${\mathbf {RP}}^2$ or the complexes $Z_2, Z_3$ shown in Figure \ref{fig6}. 
%  \item[(C)] any subcomplex $S\subset Y$ with at most $2\epsilon^{-1}$ faces is aspherical. 
\end{description}
In particular, a random 2-complex $Y\in Y(n,p)$ has the following property with probability tending to one as $n\to \infty$: 
any subcomplex $Y'\subset Y$ with $b_2(Y')=0$ is aspherical if and only if  
\begin{description}
  \item[(B')] $Y'$ contains no subcomplexes $S\subset Y'$ with at most $2\cdot \epsilon^{-1}$ faces which are homeomorphic to the 
  projective plane ${\mathbf {RP}}^2$.
 \end{description}    
\end{theorem}
\begin{proof} We first show that {\bf {(A) $\Rightarrow$ (B)}}. It is obvious that an aspherical 2-complex cannot contain as a subcomplex none of 
$S^2$ or $Z_2$, $Z_3$;
 otherwise there would be a nontrivial spherical homology class. By a theorem of W.H. Cockcroft \cite{CC}, an aspherical 2-complex cannot contain a projective plane. 

Next we show that {\bf {(B) $\Rightarrow$ (A)}} a.a.s.
Let $Y'\subset Y$ be a subcomplex where $Y\in Y(n,p)$ is random. Assume that $Y'$ contains no subcomplexes with at most $2/\epsilon$ faces which are homeomorphic to either $S^2$, ${\mathbf {RP}}^2$ or $Z_2, Z_3$, see above. 
If $Y'$ is not aspherical then by Theorem \ref{thmupper}, $M(Y')\le C$ where $C$ depends only on $\epsilon$. 
There exist finitely many isomorphism types $\{S_j\}_{j\in J}$ 
of triangulations of $S^2$ with at most $C$ faces. By Lemma \ref{containment} the probability that a random 2-complex $Y$ contains 
an image $\phi_j(S_j)$ satisfying 
$\tilde \mu(\phi_j(S_j))< 1/2+\epsilon$ tends to zero as $n\to \infty$. Thus we only have to consider the images 
of homotopically nontrivial simplicial maps $\phi_j:S_j\to Y'$ satisfying $\tilde \mu(\phi_j(S_j))\ge 1/2+\epsilon$. 

If the second Betti number of the image is nonzero, $b_2(\phi_j(S_j))\not=0$, then the image $\phi_j(S_j)$ contains a minimal cycle $Z$ satisfying $\tilde \mu(Z) \ge 1/2+\epsilon$. By Theorem \ref{lm4}, such $Z$ must be homeomorphic to one of the complexes $Z_1, Z_2, Z_3, Z_4$. For $i=1,2,3$ one has
$$\epsilon \, \le\,  \tilde \mu(Z_i)-\frac{1}{2} \, \le\,  \frac{2}{f(Z_i)}$$
and hence $f(Z_i)\le 2/\epsilon$. For $i=4$, we have $Z_4=P^2\cup \Delta^2$ and 
$$\mu(P^2) =\frac{1}{2} + \frac{1}{f(P^2)} \ge \tilde \mu(Z_4) \ge \frac{1}{2} + \epsilon$$
implying that $f(P^2)\le 2\epsilon^{-1}$. Now we may use our assumptions concerning $Y'$, implying that $Y'$ contains no subcomplexes homeomorphic to $Z_1, \dots, Z_4$ leading to contradiction. 

Consider now the remaining case $b_2(\phi_j(S_j))=0$. By Lemma \ref{containment} one has 
$$\mu(\phi_j(S_j))\ge \tilde \mu(\phi_j(S_j))\ge \frac{1}{2} +\epsilon.$$
Now, applying Lemma \ref{btwo01}, we see that the image $\phi_j(S_j)$ is an iterated wedge of projective planes. 
The argument similar to the one used in the previous paragraph shows that each of these projective planes $P^2$ has at most $2\epsilon^{-1}$ faces. Hence we obtain a contradiction to our assumption.

%The implications {\bf {(C) $\Rightarrow$ (B)} $\Rightarrow$ (C)} are obvious. 
\end{proof}

The following results are corollaries of Theorem \ref{thmasph}:

\begin{corollary} \label{whitehead} Assume that $p\ll n^{-1/2-\epsilon}$ where $\epsilon >0$ is fixed.  
Then a random 2-complex $Y\in Y(n,p)$ has the following property with probability tending to one as $n\to \infty$: 
any aspherical subcomplex $Y'\subset Y$ satisfies the Whitehead conjecture, i.e. if a subcomplex $Y'\subset Y$ is aspherical then all subcomplexes of $Y'$ are also aspherical.
\end{corollary}

\begin{proof}
Let $Y'\subset Y$ be aspherical. Then by the previous Theorem, $Y'$ has property {\bf {(B)}}. Hence any subcomplex $Y'' \subset Y'$ has property {\bf {(B)}}. 
Applying Theorem \ref{thmasph} again we obtain that $Y''$ is aspherical. This completes the proof. 
\end{proof}

\begin{corollary} Assume that $p\ll n^{-1/2-\epsilon}$ where $\epsilon >0$ is fixed.  
Then a random 2-complex $Y\in Y(n,p)$ has the following property with probability tending to one as $n\to \infty$: 
a subcomplex $Y'\subset Y$ is aspherical if and only if any subcomplex $S\subset Y'$ with at most $2\epsilon^{-1}$ faces is aspherical. 
\end{corollary}
\begin{proof} Indeed, in one direction the result follows from Corollary \ref{whitehead}. In the other direction the result follows from the implication 
{\bf (B) $\Rightarrow $ (A)} of Theorem \ref{thmasph}. 
\end{proof}

\begin{corollary} Assume that $p\ll n^{-1/2-\epsilon}$ where $\epsilon >0$ is fixed.  
Then a random 2-complex $Y\in Y(n,p)$ has the following property with probability tending to one as $n\to \infty$: 
any subcomplex $Y'\subset Y$ with $H_2(Y';\Z_2)=0$ is aspherical. 
\end{corollary}
\begin{proof}
If $H_2(Y';\Z_2)=0$ then $Y'$ cannot have subcomplexes homeomorphic to $S^2$, ${\mathbf {RP}}^2$, $Z_2$ and $Z_3$. Therefore $Y'$ is aspherical by Theorem 
\ref{thmasph}.  
\end{proof}

\begin{corollary}\label{cor35} Assume that $p\ll n^{-3/5}$. Then for a random 2-complex $Y\in Y(n,p)$ the fundamental group $\pi_1(Y)$ has cohomological dimension at most 2, a.a.s. In particular, under this assumption on $p$ the fundamental group $\pi_1(Y)$ has no torsion, a.a.s.
Moreover, for $p\ll n^{-3/5}$, any subcomplex $Y'\subset Y$ with $b_2(Y')=0$ is aspherical a.a.s. 
\end{corollary}
\begin{proof} We shall apply Theorem \ref{thmasph} with $\epsilon= 1/10$, so that $1/2+\epsilon=3/5$ and $2\epsilon^{-1}=20$. 
Thus we shall consider embeddings 
of triangulated $S^2$, $Z_2$, $Z_3$ and ${\mathbf {RP}}^2$ into $Y\in Y(n,p)$ having at most $20$ faces.  

For any triangulation of ${\mathbf {RP}}^2$ one has
$\tilde \mu({\mathbf {RP}}^2)\le 3/5$ (see \cite{CCFK}, Corollary 22) and by Lemma \ref{containment}, no projective plane with at most 20 faces is contained in $Y\in Y(n,p)$ under the assumption $p\ll n^{-3/5}$. 

Thus $Y$ may only contain copies of $S^2, Z_2, Z_3$ with at most 20 faces. We produce now a sequence of subcomplexes 
$Y=Y_0\supset Y_1\supset Y_2\supset\dots\supset Y_k$ with $k\le b_2(Y)$ as follows. Each $Y_{i+1}$ is obtained from $Y_i$ by removing a face belonging to a subcomplex homeomorphic to $S^2, Z_2$ or $Z_3$. Then clearly $\pi_1(Y_{i+1}) = \pi_1(Y_i)$ and the last subcomplex $Y_k\subset Y$ 
is aspherical by Theorem \ref{thmasph}. 

Therefore, we obtain that there exists a 2-dimensional complex $Y_k$ having fundamental group $\pi_1(Y_k)=\pi_1(Y)$. 
This implies that the cohomological dimension of $\pi_1(Y)$ is at most $2$. 

As for the second statement of the Corollary, we observe that any subcomplex $Y'\subset Y$ with $b_2(Y')=0$ cannot contain $S^2, Z_2, Z_3$. It also cannot contain 
copies of the real projective plane, as explained above. Thus, by Theorem \ref{thmasph}, $Y'$ must be aspherical, a.a.s. 
\end{proof}

\begin{remark} {\rm The method of the proof of Corollary \ref{cor35} gives in fact a stronger result: for $p\ll n^{-3/5}$ a random 2-complex $Y\in Y(n,p)$ has the following property with probability tending to one: the fundamental group of any subcomplex $Y'\subset Y$ has cohomological dimension at most $2$, i.e. 
${cd}(\pi_1(Y'))\le 2$.
}\end{remark}

Using Theorem \ref{thmasph} one may generate random aspherical 2-complexes as follows. Start with a random 2-complex $Y\in Y(n,p)$ where 
$p\ll n^{-1/2-\epsilon}$. Let $\mathcal A_\epsilon=\{Z\}$ be the set of all isomorphism types of minimal cycles having at most $2\epsilon^{-1}$ faces with the property $\mu(Z) >1/2$. The set $\mathcal A_\epsilon$ is finite; each $Z\in \mathcal A_\epsilon$ is homeomorphic to one of the 2-complexes $Z_1, Z_2, Z_3, Z_4$ described in Theorem 
\ref{lm4}. Now one finds all subcomplexes of the random complex $Y$ which are isomorphic to the complexes $Z\in \mathcal A_\epsilon$ and removes randomly (or by applying certain rule) one of its faces satisfying the conditions of Corollary \ref{cor2}. 
The obtained complex $Y'\subset Y$ is aspherical, a.a.s., by Theorem \ref{thmasph}; clearly $\pi_1(Y')=\pi_1(Y)$. 

In the next statement we show that the aspherical 2-complex $Y'$ has large second Betti number and hence the fundamental group 
$\pi_1(Y')=\pi_1(Y)$
has cohomological dimension 2.  Combined with Corollary \ref{cor35}, Proposition \ref{remcomb} implies Theorem A from the Introduction. 

\begin{proposition} \label{remcomb} Assume that for some $c>3$ and $\epsilon >0$ one has 
\begin{eqnarray}\label{estim}
\frac{c}{n} \, <\,  p \, <\,  n^{-1/2-\epsilon}.
\end{eqnarray}
Then the second Betti number of the random aspherical 2-complex $Y'$ described in the preceding paragraph satisfies
\begin{eqnarray}\label{estim1}
n^2\cdot \frac{c-3}{8}\, \le\,  b_2(Y') \, \le \, n^{5/2-\epsilon},
\end{eqnarray}
a.a.s.
\end{proposition} 
This result is similar to Theorem 3 from \cite{CF} and its proof uses the same sequence of arguments. 

%
%
%Suppose that for $Y\in Y(n,p)$ we have the tetrahedra $T_1, \dots, T_k\subset Y$ as in Definition \ref{def1}. Let $Z$ be obtained from $Y$ by removing 
%a face from each of the tetrahedra $T_i$. We want to show that $b_2(Z)>0$, a.a.s. Since $\pi_1(Z) = \pi_1(Y)$ and $Z$ is aspherical, it would imply that 
%${\mathrm {cd}}(\pi_1(Y))=2$. 

\begin{proof} 
Let $f_2, b_2: Y(n,p)\to \Z$ denote the number of 2-simplexes in a random complex and the second Betti number, viewed as random variables. 
Note that $f_2$ is binomially distributed with expectation $p\binom n 3$. We have 
$
f_2 - \binom {n-1}{2}\, \le\,  b_2\, \le\, f_2 . 
$
The inequality (2.6) on page 26 of \cite{JLR} with $t=n^{7/4}$ gives
$$\PP(f_2 \le p\binom n 3 -t) \le \exp\left(-\frac{t^2}{2p\binom n 3}\right) \le \exp\left(-\sqrt{n}\right).$$
We see that with probability at least $1- \exp\left(-\sqrt{n}\right)$, one has 
$$b_2\ge f_2 - \binom {n-1}{2}\ge p{\binom n 3}- {\binom {n-1} 2}-t \ge \binom {n-1} 2 \cdot \left(\frac{c-3}{3}\right) -t \ge n^2\cdot \frac{c-3}{7}
 $$
for large $n$; here $b_2=b_2(Y)$. 

For any $Z\in {\mathcal A}_\epsilon$ consider the random variable $k_Z: Y(n,p) \to \Z$, where for $Y\in Y(n,p)$ the value 
$k_Z(Y)$ is the number of subcomplexes of $Y$ which are isomorphic (as simplicial complexes) to $Z$.
The expectation of $k_Z$ satisfies $\E(k_Z) \le n^{v(Z)}p^{f_2(Z)}$. For any $Z\in {\mathcal A}_\epsilon$ one has $L(Z)\le 0$ and $\chi(Z) \le 2$ implying 
$$\mu(Z)=\frac{v(Z)}{f_2(Z)} \le \frac{1}{2} + \frac{2}{f_2(Z)},$$
as follows from formula (\ref{mu}). Hence we have $$v(Z) \le f_2(Z)/2 +2$$
and therefore
\begin{eqnarray*}
\E(k_Z) &\le & n^{v(Z)}p^{f_2(Z)}\\
&\le& n^{f_2(Z)/2+2}p^{f_2(Z)}\\
&\le& n^{f_2(Z)/2+2}n^{{(-1/2 -\epsilon )}f_2(Z)}\\
& =& n^{2-\epsilon f_2(Z) } \le n^{2-4\epsilon}.
\end{eqnarray*}

Let $k=\sum_{Z\in {\mathcal A}_\epsilon} k_Z$ be the sum random variable. Then we obtain
$$\E(k) =\sum_Z \E(k_Z) \le |{\mathcal A}_\epsilon|\cdot n^{2-4\epsilon}.$$
The Markov inequality $\PP(k\ge t) \le \frac{\E(k)}{t}$ with $t=n^{2-3\epsilon}$ gives
$\PP(k\ge n^{2-3\epsilon}) \le |{\mathcal A}_\epsilon|\cdot n^{-\epsilon}.$ 
Thus, with probability tending to one as $n\to \infty$, one has
$$b_2(Y') \ge b_2(Y) - k(Y) \ge n^2\cdot \frac{c-3}{7} - n^{2 -3\epsilon} \ge n^2\cdot \frac{c-3}{8}>0$$
for $n\to \infty$. This proves the left inequality in (\ref{estim1}).

From inequality (2.5) on page 26 of \cite{JLR} with $t=n^{7/4}$ one obtains that with probability at least $1- \exp(-\sqrt{n})$ one has
$f_2\, \le\,  p\binom n 3 +t $ and hence
$ b_2(Y') \le  b_2(Y) \le f_2 \le  p\binom n 3 +t \le n^{5/2-\epsilon}.$
This proves the right inequality in (\ref{estim1}).
%
%From Theorem \ref{thm1} we know that for $Y\in Y(n,p)$ one has ${\mathrm {cd}}(\pi_1(Y)) \le 2$ and from statement (A) we obtain that 
%${\mathrm {cd}}(\pi_1(Y))\ge 2$ a.a.s., since the group $\pi_1(Y)$ has nontrivial 2-dimensional homology. 
%This implies statement (B). 
\end{proof}

As above, for a positive integer $m$, let $M_m$ denote the Moore surface, 
i.e. the quotient space of the disc $D^2\subset \C$ (the unit disc on the complex plane) where every point $z\in \partial D^2$ is identified with $z^m$.
\begin{corollary} If $p\ll n^{-1/2-\epsilon}$ then a random 2-complex $Y\in Y(n,p)$ with probability tending to one as $n\to\infty$ has the following property: 
no subcomplex $Y'\subset Y$ homeomorphic to a Moore surface $M_m$ with $m\ge 3$. 
\end{corollary}
\begin{proof}
This is a corollary of Theorem \ref{thmasph} since for $m\ge 3$ the Moore surface $M_m$ (a) contains no subcomplexes homeomorphic to the sphere, to the real projective space and to complexes $Z_2, Z_3$ (shown in Figure \ref{fig6}) and (b) the Moore surface $M_m$ is not aspherical. 
\end{proof}

A similar argument provides an alternative proof of Corollary \ref{noprojplanes}.

%
%\begin{remark} \label{remcomb} {\rm Combining the previous Corollary with the arguments of the proof of Theorem 3 from \cite{CF}, one obtains that the cohomological dimension of a random group 
%$\pi_1(Y)$ (where $Y\in Y(n,p)$) equals 2 assuming that 
%$\frac{c}{n} < p\ll n^{-3/5},$
%for any $c>3$. 
%}
%\end{remark} 
\vskip 2cm
\section{Appendix: Proof of Theorem \ref{hyp}}\label{app}

The proof of Theorem \ref{hyp} given below is similar to the arguments of \cite{BHK} and is based on two auxiliary results: (1) the local-to-global principle of Gromov \cite{Gromov87} and on (2) Theorem \ref{uniform} giving uniform isoperimetric constants for complexes satisfying $\tilde \mu(X)\ge 1/2+\epsilon$.

The local-to-global principle of Gromov can be stated as follows:

\begin{theorem}\label{localtoglobal} Let $X$ be a finite 2-complex and let $C>0$ be a constant such that any pure subcomplex $S\subset X$ having at most 
$(44)^3\cdot C^{-2}$ two-dimensional simplexes satisfies $I(S)\ge C$. Then $I(X)\ge C\cdot 44^{-1}$. 
\end{theorem}

Theorem \ref{localtoglobal} follows from Theorem 3.9 from \cite{BHK}. Indeed, suppose that the assumptions of Theorem \ref{localtoglobal} are satisfied. 
Let $\gamma: S^1\to X$ be a simplicial loop with $A_X(\gamma)<(44)^3\cdot C^{-2}= 44\rho^2$ where $\rho=44/C$. Then there is a pure subcomplex $S\subset X$ 
with at most $(44)^3\cdot C^{-2}$ faces, which contains $\gamma$ and the minimal spanning disc for $\gamma$ in $X$, 
such that $$\frac{|\gamma|}{A_S(\gamma)}\ge I(S) \ge C=44/\rho,$$ by our assumption, i.e. 
$$A_X(\gamma) =A_S(\gamma) \le \frac{\rho}{44} \cdot |\gamma|.$$ 
Applying Theorem 3.9 from \cite{BHK} we obtain $I(X) \ge \rho^{-1}= C/44$.

Let $X$ be a 2-complex satisfying $\tilde \mu(X) >1/2$. Then by Corollary \ref{thm6} the fundamental group of $X$ is hyperbolic as it is a free product of several copies of 
cyclic groups $\Z$ and $\Z_2$. Hence, $I(X)>0$. The following theorem gives a uniform lower bound for the numbers $I(X)$.

\begin{theorem}\label{uniform}
Given $\epsilon >0$ there exists a constant $C_\epsilon>0$ such that for any finite pure 2-complex $X$ with $\tilde \mu(X)\ge 1/2 + \epsilon$ one has $I(X) \ge C_\epsilon$. 
\end{theorem}

%\noindent{\bf Observation:} If $X\subset Y$ and $\pi_1(X) \to \pi_1(Y)$ is injective then $$I(X)\le I(Y).$$

This statement is equivalent to Lemma 3.5 from \cite{BHK} and plays a crucial role in the proof of Theorem \ref{hyp}. 
A proof of Theorem \ref{uniform} (which uses the method of \cite{BHK} but is presented in a slightly different form) is given below 
in \S \ref{secthm10}. 

\begin{proof}[Proof of Theorem \ref{hyp} using Theorem \ref{localtoglobal} and Theorem  \ref{uniform}] Let $C_\epsilon$ be the constant given by Theorem
 \ref{uniform}. 
Consider the set $\mathcal S$ of isomorphism types of all pure 2-complexes having at most $44^3\cdot C_\epsilon^{-2}$ faces. Clearly, the set $\mathcal S$ is finite. We may present it as the disjoint union $\mathcal S= \mathcal S_1 \sqcup \mathcal S_2$ where any $S\in \mathcal S_1$ satisfies $\tilde \mu(S)\ge 1/2 +\epsilon$ while for 
$S\in \mathcal S_2$ one has $\tilde \mu(S) < 1/2 +\epsilon$. By Lemma \ref{containment}, a random 2-complex $Y\in Y(n, p)$ contain as subcomplexes complexes  $S\in \mathcal S_2$ with probability tending to zero as $n\to \infty$. Hence, $Y\in Y(n,p)$ may contain as subcomplexes only complexes $S\in \mathcal S_1$, a.a.s. 
By Theorem \ref{uniform}, any $S\in \mathcal S_1$ satisfies $I(S)\ge C_\epsilon$. Hence we see that with probability tending to one, any subcomplex $S$ of $Y$ having at most $44^3\cdot C_\epsilon^{-2}$ faces satisfies $I(S)\ge C_\epsilon$. 
Now applying Theorem \ref{localtoglobal} we obtain $I(Y')\ge 
C_\epsilon\cdot 44^{-1}=c_\epsilon$, for any subcomplex $Y'\subset Y$, a.a.s.
\end{proof}

\subsection{Proof of Theorem \ref{uniform}}\label{secthm10}

\begin{definition}
We will say that a finite 2-complex $X$ is tight if for any proper subcomplex $X'\subset X$, $X'\not= X$, one has
$I(X') > I(X).$
\end{definition}
Clearly, one has 
\begin{eqnarray}\label{ineq}
I(X) \ge \min \{I(Y)\} %; \mbox{where $X'\subset X$ is a proper tight subcomplex}\}.
\end{eqnarray}
where $Y\subset X$ is a tight subcomplex. Since $\tilde \mu(Y) \ge \tilde \mu(X)$ for $Y\subset X$, it is obvious from (\ref{ineq}) that it is enough to prove Theorem \ref{uniform} under the additional assumption that $X$ is tight. 
%\begin{eqnarray}
%I(X) \ge \min \{I(X'); \mbox{where $X'\subset X$ is a proper tight subcomplex}\}.
%\end{eqnarray}

\begin{remark}\label{rmrk12} {\rm Suppose that $X$ is pure and tight and suppose that $\gamma: S^1\to X$ is a simplicial loop with $|\gamma|\cdot A_X(\gamma)^{-1}$ less than the minimum of the numbers $I(X')$ where $X'\subset X$ is a proper subcomplex. Let $b: D^2\to X$ be a minimal spanning disc for $\gamma$; then $b(D^2)=X,$ i.e. $b$ is surjective. Indeed, if the image of $b$ does not contain a 2-simplex $\sigma$ then removing 
it we obtain a subcomplex $X'\subset X$ with $A_{X'}(\gamma)=A_X(\gamma)$ and hence $I(X') \le |\gamma|\cdot A_X(\gamma)^{-1}$ contradicting the assumption on $\gamma$. }
\end{remark}

%{\bf Question:} Let $X$ be simply connected and tight. Is it true that $\partial X$ belongs to the image of any simple simplicial loop $\gamma$ satisfying $|\gamma|\cdot A_X(\gamma)^{-1} = I(X)$?%\marginpar{Armindo, I think this must be true but the precise proof seem to be elusive.}

\begin{lemma}\label{lm13} If $X$ is a tight complex with $\tilde \mu(X)>1/2$ then $b_2(X)=0$. 
\end{lemma}
\begin{proof} Assume that $b_2(X)\not=0$. Then there exists a minimal cycle $Z\subset X$ satisfying $\mu(Z)>1/2$. Hence, by Corollary \ref{cor2}, we may find a 2-simplex $\sigma\subset Z\subset X$
such that $\partial \sigma$ is null-homotopic in $Z-\sigma\subset X-\sigma=X'$. Note that $X'^{(1)}=X^{(1)}$ and a simplicial curve $\gamma: S^1\to X'$ is null-homotopic in $X'$ if and only if it is null-homotopic in $X$. Besides, $A_X(\gamma) \le A_{X'}(\gamma)$ and hence
$$\frac{|\gamma|}{A_X(\gamma)}\ge \frac{|\gamma|}{A_{X'}(\gamma)},$$
which implies that $I(X) \ge I(X') >I(X)$ -- contradiction. 
\end{proof}

%
%\begin{example} {\rm
%Let $X$ be a triangulated disk. Then 
%$$\frac{1}{2} +\frac{1}{2}I(X) \le \tilde \mu(X) \le \frac{1}{2} + I(X).$$
%Besides, a disk $X$ is tight iff $X$ is balanced. } We will not use this remark in this paper and will leave it without proof. 
%\end{example}

\begin{lemma}\label{lneg} Given $\epsilon >0$ there exists a constant $C'_\epsilon>0$ such that for any finite pure tight connected 2-complex with $\tilde \mu(X) \ge 1/2+\epsilon$ and $L(X) \le 0$ one has $I(X) \ge C'_\epsilon$. 
\end{lemma}
This Lemma is similar to Theorem \ref{uniform} but it has an additional assumption that $L(X) \le 0$. It is clear from the proof that the assumption $L(X)\le 0$ can be replaced by any assumption of the type $L(X) \le 1000$, i.e. by any specific upper bound, without altering the proof. 
\begin{proof} We show that the number of isomorphism types of complexes $X$ satisfying the conditions of the Lemma is finite; hence the statement of the Lemma follows  by setting $C'_\epsilon=\min I(X)$ and using Corollary \ref{thm6} which gives $I(X)>0$ (since $\pi_1(X)$ is hyperbolic) and hence $C'_\epsilon>0$. 
The inequality 
$$\mu(X) = \frac{1}{2} +\frac{2\chi(X)+L(X)}{2f(X)} \ge \frac{1}{2} + \epsilon$$
is equivalent to 
$$f(X) \le \epsilon^{-1}\cdot (\chi(X) +L(X)/2),$$
where $f(X)$ denotes the number of 2-simplexes in $X$. 
By Lemma \ref{lm13} we have $\chi(X) =1-b_1(X) \le 1$ and using the assumption $L(X) \le 0$ we obtain
$f(X) \le \epsilon^{-1}.$
This implies the finiteness of the set of possible isomorphism types of $X$ and completes the proof. 
\end{proof}

We will also use a relative isoperimetric constant $I(X, X')\in \R$ for a pair consisting of a finite 2-complex 
$X$ and its subcomplex $X'\subset X$ which will be defined as the infimum of all ratios 
$
{|\gamma|}\cdot{A_X(\gamma)}^{-1}$
where $\gamma: S^1\to X'$ runs over simplicial loops in $X'$ which are null-homotopic in $X$.  
Clearly, $I(X,X')\ge I(X)$ and $I(X, X')=I(X)$ if $X'=X$.
Below is a useful strengthening of Lemma \ref{lneg}.

\begin{lemma}\label{lnegs} Given $\epsilon >0$, let $C'_\epsilon>0$ be the constant given by Lemma \ref{lneg}. Then for any finite pure tight connected 2-complex with $\tilde \mu(X) \ge 1/2+\epsilon$ and for a connected subcomplex $X'\subset X$ satisfying $L(X') \le 0$ one has $I(X,X') \ge C'_\epsilon$. 
\end{lemma}

\begin{proof} 
We show below that under the assumptions on $X$, $X'$ one has 
\begin{eqnarray} \label{y}
I(X,X') \ge \min_Y I(Y)
\end{eqnarray}
where $Y$ runs over all subcomplexes $X'\subset Y\subset X$ satisfying $L(Y)\le 0$. Clearly, $\tilde \mu(Y)\ge 1/2+\epsilon$ for any such $Y$. 
 In the proof of Lemma \ref{lneg} we showed that $b_2(X)=0$ which implies that $b_2(Y)=0$. Besides, without loss of generality we may assume that 
$Y$ is connected. The arguments of the proof of Lemma \ref{lneg} now apply (i.e. $Y$ may have finitely many isomorphism types, each having a hyperbolic fundamental group) and the result follows $I(X,X')\ge \min_Y I(Y)\ge C'_\epsilon$. 

Suppose that inequality (\ref{y}) is false, i.e. $I(X,X') < \min_Y I(Y)$, and consider a simplicial loop $\gamma:S^1\to X'$ satisfying $\gamma\sim 1$ in $X$ and 
$|\gamma|\cdot A_X(\gamma)^{-1} <\min_Y I(Y).$ Let $\psi: D^2\to X$ be a simplicial spanning disc of minimal area. 
It follows from the arguments of Ronan \cite{Ron}, that $\psi$ is non-degenerate in the following sense: for any 2-simplex $\sigma$ of $D^2$ the image $\psi(\sigma)$ is a 2-simplex 
and for two distinct 2-simplexes $\sigma_1, \sigma_2$ of $D^2$ with $\psi(\sigma_1)=\psi(\sigma_2)$ the intersection $\sigma_1\cap \sigma_2$ is either $\emptyset$ or a vertex of $D^2$. In other words, we exclude {\it foldings}, i.e. situations such that $\psi(\sigma_1)=\psi(\sigma_2)$ and $\sigma_1\cap \sigma_2$ is an edge.  
Consider $Z=X'\cup \psi(D^2)$. Note that $L(Z)\le 0$. Indeed, since $$L(Z)=\sum_e (2-\deg_Z(e)),$$ where $e$ runs over the edges of $Z$, we see that for $e\subset X'$,
$\deg_{X'}(e)\le \deg_Z(e)$ and for a newly created edges $e\subset \psi(D^2)$, clearly $\deg_Z(e)\ge 2$. Hence, $L(Z)\le L(X')\le 0$. 
On the other hand, $A_X(\gamma)=A_Z(\gamma)$ and hence $I(Z)\le |\gamma|\cdot A_X(\gamma)^{-1}<\min_Y I(Y)$, a contradiction.
\end{proof}

The main idea of the proof of Theorem \ref{uniform} in the general case is to find a planar complex (a \lq\lq singular  surface\rq\rq) \, $\Sigma$ with one boundary component $\partial_+\Sigma$ being the initial loop and such that \lq\lq the rest of the boundary\rq\rq\, $\partial_-\Sigma$ is a \lq\lq product of negative loops\rq\rq\,  (i.e. loops satisfying Lemma \ref{lnegs}). The essential part of the proof is in estimating the area (the number of 2-simplexes) of such $\Sigma$.

\begin{proof}[Proof of Theorem \ref{uniform}]
Consider a connected tight pure 2-complex $X$ satisfying
\begin{eqnarray}\label{ass}
\tilde \mu(X) \ge \frac{1}{2} +\epsilon
\end{eqnarray}
and a simplicial prime loop $\gamma: S^1 \to X$ such that the ratio 
$|\gamma|\cdot A_X(\gamma)^{-1}$ is less than the minimum of the numbers $I(X')$ for all proper subcomplexes $X'\subset X$. Consider a minimal spanning disc 
$b: D^2\to X$ for $\gamma=b|_{\partial D^2}$; here $D^2$ is a triangulated disc and $b$ is a simplicial map. As we showed in Remark \ref{rmrk12}, the map $b$ is surjective. 
As explained in the proof of Lemma \ref{lnegs}, due to arguments of Ronan \cite{Ron}, 
we may assume that $b$ has no foldings. 

%It follows from the arguments of , that $b$ is non-degenerate in the following sense: for any 2-simplex $\sigma$ of $D^2$ the image $b(\sigma)$ is a 2-simplex 
%and for two distinct 2-simplexes $\sigma_1, \sigma_2$ of $D^2$ with $b(\sigma_1)=b(\sigma_2)$ the intersection $\sigma_1\cap \sigma_2$ is either $\emptyset$ or a vertex of $D^2$. In other words, we exclude foldings, i.e. situations such that $b(\sigma_1)=b(\sigma_2)$ and $\sigma_1\cap \sigma_2$ is an edge. 

For any integer $i\ge 1$ we denote by $X_i\subset X$ the pure subcomplex generated by all 2-simplexes $\sigma$ of $X$ such that the preimage $b^{-1}(\sigma)\subset D^2$ contains $\ge i$ two-dimensional simplexes. One has $X=X_1\supset X_2\supset X_3\supset \dots.$ Each $X_i$ may have several connected components and we will denote by $\Lambda$ the set labelling all the connected components of the disjoint union $\sqcup_{i\ge 1} X_i$. For $\lambda\in \Lambda$ the symbol $X_\lambda$ will denote the corresponding connected component of $\sqcup_{i\ge 1} X_i$
and the symbol 
$i=i(\lambda)\in \{1, 2, \dots\}$ will denote the index $i\ge 1$ such that $X_\lambda$ is a connected component of $X_i$, viewed as a subset of $\sqcup_{i\ge 1} X_i$. We endow $\Lambda$ with the following partial order: $\lambda_1\le \lambda_2$ iff $X_{\lambda_1}\supset X_{\lambda_2}$ (where $X_{\lambda_1}$ and $X_{\lambda_2}$ are viewed as subsets of $X$) and $i(\lambda_1)\le i(\lambda_2)$. 

Next we define the sets
$$\Lambda^-=\{\lambda\in \Lambda; L(X_\lambda)\le 0\}$$
and 
$$ \Lambda^+=\{\lambda\in \Lambda; \mbox{for any $\mu\in \Lambda$ with $\mu\le \lambda$, }\, L(X_\mu)> 0\}.$$
Finally we consider the following subcomplex  of the disk $D^2$:
\begin{eqnarray}\label{defsigma}
\Sigma' =D^2-\bigcup_{\lambda\in \Lambda^-}{\rm {Int}}(b^{-1}(X_\lambda))\end{eqnarray}
and we shall denote by $\Sigma$ the connected component of $\Sigma'$ containing the boundary circle $\partial D^2$. 
%
%and $D^-$ is defined as the smallest subcomplex of $D^2$ containing $D^2-D^+$ and such that each its connected component is contractible. In general, the complex 
%$D^2-D^+$ may have several \lq\lq holes\rq\rq\ and the complex $D^-$ is obtained from $D^2-D^+$ by filling these holes by discs.  

Recall that for a 2-complex $X$ the symbol $f(X)$ denotes the number of 2-simplexes in $X$. We have 
\begin{eqnarray}\label{one1}
f(D^2) = \sum_{\lambda\in \Lambda} f(X_\lambda),
\end{eqnarray}
and 
\begin{eqnarray}\label{two2}
f(\Sigma) \le f(\Sigma') \le \sum_{\lambda\in \Lambda^+} f(X_\lambda).
\end{eqnarray}
Formula (\ref{one1}) follows from the observation that any 2-simplex of $X=b(D^2)$ contributes to the RHS of (\ref{one1}) as many units as its multiplicity (the number of its preimages under $b$). Formula (\ref{two2}) follows from (\ref{one1}) and from the fact that for a 2-simplex 
$\sigma$ of $\Sigma'$ the image $b(\sigma)$ lies always in the complexes $X_\lambda$ with $L(X_\lambda)> 0$.  
\begin{figure}[h]\label{fig2}
\centering
\includegraphics[width=0.4\textwidth]{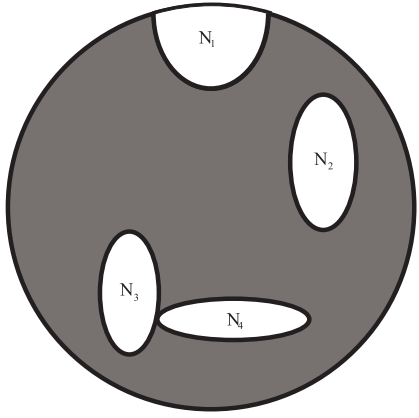}
\caption{The complex $\Sigma\subset D^2$.}
\end{figure}

\begin{lemma} \label{lpartial} One has the following inequality
\begin{eqnarray}\label{llb}
\sum_{\lambda\in \Lambda^+}L(X_\lambda)\le |\partial D^2|.
\end{eqnarray}
\end{lemma}
\begin{proof} For an edge $e$ of $X$ and for $\lambda\in \Lambda$ we denote 
$$\widetilde \deg_\lambda(e) = \left\{ \begin{array}{lll}
2-\deg_{X_\lambda}(e), &\mbox{if}& e\subset X_\lambda,\\
0, &\mbox{if}& e\not \subset X_\lambda.\end{array}\right. 
$$
Note that $\widetilde \deg_\lambda(e) \le 1$ and $\widetilde \deg_\lambda(e) = 1$ iff $e$ belongs to a unique 2-simplex of 
$X_\lambda$, i.e. $e\subset \partial X_\lambda$. 
One has
\begin{eqnarray}
\sum_{\lambda\in \Lambda^+}L(X_\lambda) = \sum_{e\subset X}\sum_{\lambda\in \Lambda^+}\widetilde{\deg}_\lambda (e)&\le &
\sum_{e\subset X} \max\left\{\sum_{\Lambda^+} \widetilde \deg_\lambda(e), 0\right\}\nonumber \\ &\le &
\sum_{e\subset X} \max\left\{\sum_{\Lambda} \widetilde \deg_\lambda(e), 0\right\}\label{eight} \\ 
&\le & \sum_{e\subset X} |b^{-1}(e)\cap \partial D^2|   = |\partial D^2|.\label{nine}
\end{eqnarray}
Here $|b^{-1}(e)\cap \partial D^2|$ denotes the number of boundary edges which $b$ maps onto the edge $e$.
The inequality (\ref{eight}) follows from the fact that $\lambda\mapsto \widetilde \deg_\lambda(e)$ is a non-decreasing function on the poset 
$\Lambda_e=\{\lambda\in \Lambda; e\subset X_\lambda\}$. Note that $\Lambda_e$ is in fact linearly ordered. Indeed, suppose that 
$\lambda_1, \lambda_2\in \Lambda_e$ are such that $i_1=i(\lambda_1) \ge i_2=i(\lambda_2)$. Then $X_{\lambda_2}$, as a path-component of $X_{i_2}$, is contained in a path-component of $X_{i_1}$; the latter must coincide with $X_{\lambda_1}$ 
since the intersection $X_{\lambda_1}\cap X_{\lambda_2}\not=\emptyset$ is nonempty (since it contains the edge $e$). Thus we see that the assumption $i(\lambda_1)\ge i(\lambda_2)$ implies that 
$X_{\lambda_2}\subset X_{\lambda_1}$, i.e. $\lambda_1\ge \lambda_2$. 
Denoting $\Lambda_e\cap \Lambda^+=\Lambda_e^+$ we observe that if the sum
\begin{eqnarray}\label{pos}
\sum_{\Lambda^+}\widetilde\deg_\lambda(e) =\sum_{\Lambda^+_e}\widetilde\deg_\lambda(e)  >0
\end{eqnarray}
is positive then (due to the monotonicity) one has
$\widetilde \deg_{\lambda_0}(e)>0$
for the maximal element $\lambda_0$ of $\Lambda_e^+$ and hence $\widetilde \deg_\lambda(e)>0$ for all 
$\lambda \ge \lambda_0$, $\lambda\in \Lambda_e$. Thus, the positivity (\ref{pos}) implies $\sum_{\Lambda^+} \widetilde\deg_\lambda(e) < \sum_{\Lambda} \widetilde\deg_\lambda(e)$ giving (\ref{eight}). 

To prove the inequality (\ref{nine}) we observe that for an edge $e$ of $X$ one has
\begin{eqnarray}\label{11}
\sum_{\lambda\in \Lambda}\deg_{X_\lambda}(e) = \sum_{e'\in b^{-1}(e)}\deg_{D^2}(e') = 2|b^{-1}(e)|-|b^{-1}(e)\cap \partial D^2|
\end{eqnarray}
and hence 
\begin{eqnarray}
\sum_{\lambda\in \Lambda}\widetilde\deg_\lambda(e) &=& 2\cdot |\Lambda^e| - \sum_{\lambda\in \Lambda}\deg_\lambda(e)\nonumber  \\ &\le &
2\cdot |b^{-1}(e)| - \sum_{\lambda\in \Lambda}\deg_\lambda(e) \label{12}\\ & =& |b^{-1}(e)\cap \partial D^2| \, \, \, \mbox{(using (\ref{11}))}\nonumber.
\end{eqnarray}
Inequality (\ref{12}) is based on
\begin{eqnarray}\label{13}
|\Lambda_e| \, = \, \max_{\sigma\supset e} |b^{-1}(\sigma)| \, \le \, |b^{-1}(e)|,\end{eqnarray}
where $\sigma$ runs over all 2-simplexes of $X$ containing $e$. The equality which appears in (\ref{13}) is a consequence of the set $\Lambda_e$ being linearly ordered, see above.
This completes the proof. 
\end{proof}

%
%which follows from the existence of an injective map
%$$\psi: \Lambda_e\to b^{-1}(e), \quad \lambda \mapsto e'$$ 
%defined as follows. Fix a linear order $\ge $ of 2-simplexes of $D^2$ satisfying the following condition: 
%if $\sigma_1\ge \sigma_2$ then $|b^{-1}(b(\sigma_1))| \ge |b^{-1}(b(\sigma_2))|$. 
%For $\lambda\in \Lambda_e$ consider the largest simplex $\sigma\subset D^2$ with $e\subset b(\sigma)\subset X_\lambda$. 
%We set $e'=\psi(e)$ to be the unique face of $\sigma$ with $b(e')=e$. 
%The map $\psi$ is injective since, given an edge $e'=\psi(\lambda)$ we may recover  $\lambda\in \Lambda_e$ as follows: 
%find the largest simplex $\sigma$ incident to $e'$ and consider the maximal $\lambda\in \Lambda_e$ satisfying 
%$b(\sigma)\subset X_\lambda$. 

Now we continue with the proof of Theorem \ref{uniform}.
Consider a tight pure 2-complex $X$ satisfying (\ref{ass}) and a simplicial loop $\gamma: S^1\to X$ as above. 
We will use the notation introduced earlier. The complex $\Sigma$ is a connected subcomplex of the disk $D^2$; it contains the boundary circle $\partial D^2$ 
which we will denote also by 
$\partial_+\Sigma$. The closure of the complement of $\Sigma$,
$$N=\overline{D^2-\Sigma}\subset D^2$$ is a pure 2-complex. Let $N=\cup_{j\in J}N_j$ be the strongly connected components of $N$.
Each $N_j$ is PL-homeomorphic to a disc and we define 
$$\partial_-\Sigma=\cup_{j\in J}\partial N_j,$$ the union of the circles $\partial N_j$ which are the boundaries of the strongly connected components of $N$. 
It may happen that $\partial_+\Sigma$ and $\partial_-\Sigma$ have nonempty intersection. Also, the circles forming $\partial_-\Sigma$ may not be disjoint. 

We claim that for any $j\in J$ there exists $\lambda\in \Lambda^-$ such that $b(\partial N_j)\subset X_\lambda$. 
Indeed, let $\lambda_1, \dots, \lambda_r\in \Lambda^-$ be the minimal elements of $\Lambda^-$ with respect to the partial order introduced earlier. The complexes
$X_{\lambda_1}, \dots, X_{\lambda_r}$ are connected and pairwise disjoint and for any $\lambda\in \Lambda^-$ the complex $X_\lambda$ is a subcomplex 
of one of the 
sets $X_{\lambda_i}$, where $i=1, \dots, r$. From our definition (\ref{defsigma}) it follows that the image of the circle $b(\partial N_j)$ is contained in the union 
$\cup_{i=1}^r X_{\lambda_i}$ but since $b(\partial N_j)$ is connected it must lie in one of the sets $X_{\lambda_i}$.

We may apply Lemma \ref{lnegs} to each of the circles $\partial N_j$. We obtain that each of the circles $\partial N_j$ admits a spanning discs of area
$\le K_\epsilon |\partial N_j|$, where $K_\epsilon= C'^{-1}_\epsilon$ is the inverse of the constant given by Lemma \ref{lnegs}. 
Using the minimality of the disc $D^2$ we obtain that the circles $\partial N$ bound in $D^2$ several discs with 
the total area 
$A \le K_\epsilon\cdot |\partial_-\Sigma|$ (here we use Lemma \ref{lneg}); 
otherwise one could reduce the area of the initial disc.

%Let us show that the number of 2-simplexes in $D^+$ satisfies
%$$f(D^+) \le \frac{3}{2\epsilon}\cdot |\partial D^2|.$$
%Indeed, since f

For $\lambda\in \Lambda^+$ one has $L(X_\lambda)\ge 1$ and $\chi(X_\lambda)\le 1$ (since $b_2(X_\lambda)=0$), hence we have 
$$3L(X_\lambda) \ge 2\chi(X_\lambda) +L(X_\lambda) \ge 2\epsilon f(X_\lambda)$$
where on the last stage we used the inequality $\mu(X_\lambda) \ge 1/2+\epsilon$. Summing up we get
$$f(\Sigma)\le  \sum_{\lambda\in \Lambda^+} f(X_{\lambda}) \le \frac{3}{2\epsilon}\sum_{\lambda\in \Lambda^+}L(X_\lambda) \le
 \frac{3}{2\epsilon}
 |\partial D^2|.$$
The rightmost inequality  is given by Lemma \ref{lpartial}. 

Next we observe, that 
\begin{eqnarray}
|\partial_-\Sigma| \le 2f(\Sigma) +|\partial_+\Sigma|.
\end{eqnarray}
To explain this inequality we note that each edge of $\partial_-\Sigma$ which does not belong to $\partial_+\Sigma$ is incident to a face of $\Sigma$ and a face of $\Sigma$ can have at most two edges lying on $\partial_-\Sigma$. 
Therefore, we obtain
\begin{eqnarray*}
f(D^2) &\le& f(\Sigma)+A \, \le \, \frac{3}{2\epsilon} |\gamma| + K_\epsilon\cdot 2\cdot f(\Sigma) + K_\epsilon |\gamma| \\ 
&\le &
\left(\frac{3}{2\epsilon}(1+2K_\epsilon) +K_\epsilon\right)\cdot |\gamma|,
\end{eqnarray*}
implying 
\begin{eqnarray}
I(X) \ge \frac{2\epsilon}{3+6K_\epsilon+2\epsilon K_\epsilon}.
\end{eqnarray}
This completes the proof of Theorem \ref{uniform}.
\end{proof}

%\begin{lemma} \label{minus}
%There exists a constant $K=K_\epsilon$ depending only on $\epsilon$ such that 
%\begin{eqnarray}
%f(D^-)\le K\cdot |\partial D^-| \le K\cdot \left(|\partial D|+3f(D^+)\right).
%\end{eqnarray}
%\end{lemma}

\bibliographystyle{amsalpha}

\begin{thebibliography}{99}


\bibitem{ALLM} L. \ Aronshtam, N. \ Linial, T. \ {\L}uczak, R. \ Meshulam, \textit{Collapsibility and vanishing of top homology in random simplicial complexes}, 
Discrete Comput. Geom.  49  (2013),  no. 2, 317–334.


\bibitem{BHK} E.\ Babson, C.\ Hoffman, M.\ Kahle, 
{\it The fundamental group of random $2$-complexes}, J. Amer. Math. Soc. 24 (2011), 1-28. See also the latest archive version arXiv:0711.2704 revised on 20.09.2012. 

\bibitem{B} W.\ A.\ Bogley, \textit{J.H.C. Whitehead's asphericity question}, in: "Two-dimensional Homotopy and Combinatorial Group Theory", eds. 
C. Hog-Angeloni, A. Sieradski and W. Metzler, LMS Lecture Notes 197, Cambridge Univ Press (1993), 309-334. 

\bibitem{CC} W.H. Cockcroft, \textit{On two-dimensional aspherical complexes}, Proc. London Math. Soc. {\bf {4}}(1954), 375-384. 

\bibitem{CCFK} D. Cohen, A.E. Costa, M. Farber, T. Kappeler, \textit{Topology of random 2-complexes,} 
Journal of Discrete and Computational Geometry, {\bf 47}(2012), 117-149.

\bibitem{CF} A.E. Costa, M. Farber, \textit{The asphericity of random 2-dimensional complexes}, to appear in \textit{Random Structures and Algorithms,}
 preprint 
arXiv:1211.3653v1

\bibitem{ER} P.\ Erd\H{o}s, A.\ R\'enyi, {\it On the evolution of 
random graphs}, Publ.\ Math.\ Inst.\ Hungar.\ Acad.\ Sci.\ {\bf 5}
(1960), 17--61.

\bibitem{Gromov83} M. Gromov, \textit{Filling Riemannian manifolds}, J. Differential Geometry, 18(1983), 1--147. 

\bibitem{Gromov87} M. Gromov, \textit{Hyperbolic groups}, in Essays in group theory, ed.
S. M. Gersten, Springer (1987), 75-265.

\bibitem{JLR} S. Janson, T. {\L}uczak, A. Ruci\'nski, \textit{Random graphs}, Wiley-Intersci. Ser. Discrete Math. Optim., Wiley-Interscience, New York, 2000.

\bibitem{KRS} M.G. Katz, Y.B. Rudyak, S. Sabourau, \textit{Systoles of 2-complexes, Reeb graphs, and Grushko decomposition}, IMRN 2006, pp.1--30. 

\bibitem{Ko} D. Kozlov, \textit{The threshold function for vanishing of the~top homology group of random $d$-complexes},
Proc. Amer. Math. Soc. 138 (2010), 4517-4527.  %\arxiv{0904.1652}.

\bibitem{LM} N.\ Linial, R.\ Meshulam, {\it Homological connectivity
  of random $2$-complexes}, Combinatorica {\bf 26} (2006),  475--487.
%\MRh{2260850}

\bibitem{MW} R.\ Meshulam, N.\ Wallach, {\it Homological
  connectivity of random $k$-complexes}, Random Structures \& Algorithms 
  \textbf{34} (2009), 408--417. 


\bibitem{Pap} P. \ Papasoglu, {\it Cheeger constants of surfaces and isoperimetric inequalities}, Transactions of the AMS {\bf {361}}(2009), pp. 5139-5162. 

%\bibitem{Pu} P.M. Pu, {\it Some inequalities in certain nonorientable Riemannian manifolds}, Pacific J. Math. {\bf 2}(1952), 55-71. 

\bibitem{Ron} M.A. Ronan, \textit{On the second homotopy group of certain simplicial complexes and some combinatorial applications}, Quart. J. Math. {\bf {32}}(1981), 225 - 233. 


\bibitem{R} S. Rosenbrock, \textit{The Whitehead Conjecture - an overview}, Siberian Electronoc Mathematical Reports, {\bf 4}(2007), 440-449. 


%\bibitem{Tutte} W.T. Tutte, \textit{A census of planar triangulations}, Canad. J. Math. 14 1962 21–38.

\end{thebibliography}

\vskip 2cm

A.E. Costa, Warwick Mathematics Institute, Warwick University, Coventry CV4 7AL, UK

email: armindocosta@gmail.com

\vskip 1cm
M. Farber, Warwick Mathematics Institute, Warwick University, Coventry CV4 7AL, UK 

email: michaelsfarber@gmail.com

    % Enter subsection title between curly braces

% Set the ending of a LaTeX document
\end{document}